\documentclass[10pt,a4paper]{amsart}

\usepackage{amsthm,amssymb,amsfonts,amsopn,latexsym,amscd,mathrsfs}
\usepackage{graphicx} 
\usepackage[all]{xy}
\usepackage[dvips]{epsfig}
\usepackage{color}

\numberwithin{figure}{section}
\xyoption{dvips}
\xyoption{rotate}

\numberwithin{equation}{section}

\newtheorem{thm}{Theorem}[section]
\newtheorem{prop}{Proposition}
\newtheorem{lm}[thm]{Lemma}
\newtheorem{crl}[thm]{Corollary}

\theoremstyle{remark}
\newtheorem{nt}[thm]{Notation}
\newtheorem{qu}[thm]{Question}

\theoremstyle{definition}
\newtheorem{ex}[thm]{Example}
\newtheorem{rem}[thm]{Remark}
\newtheorem{imprem}[thm]{Important Remark}
\newtheorem{df}[thm]{Definition}

\newcommand{\FF}{\mathcal{F}}
\newcommand{\HF}{\widehat{\FF}}
\newcommand{\MM}{\mathcal{M}}

\newcommand{\coat}{\textup{coat}} 

\newcommand{\rkl}{\sqsubset}
\newcommand{\rkleq}{\sqsubseteq}

\newcommand{\rcal}{\prec}

\newcommand{\hP}{\widehat{P}}
\newcommand{\lexl}{\lhd}
\newcommand{\lexg}{\rhd}

\newcommand{\lexldot}{\mathrel{\mbox{$\lhd\!\!\!\cdot\,$}}}
\newcommand{\lexgdot}{\mathrel{\mbox{$\rhd\!\!\!\!\cdot\,\,$}}} 
\newcommand{\texg}{\vdash} 
\newcommand{\texl}{\dashv} 
\newcommand{\tl}{\prec} 
\newcommand{\tleq}{\preccurlyeq} 
\newcommand{\tgdot}{\mathrel{\mbox{$\cdot\!\!\!\tg$}}}
\newcommand{\tg}{\succ} 
\newcommand{\Tl}{\tl} 
\newcommand{\F}{\mathcal{F}}
\newcommand{\supp}{\textup{supp}}
\renewcommand{\SS}{\mathcal{S}}
\newcommand{\tldot}{\mathrel{\mbox{$\tl\!\!\!\cdot$}}}

\newcommand{\LL}{\mathcal{L}}

\newcommand{\tlexl}{\texl^\ell}

\newcommand{\M}{\mathfrak{M}}
\newcommand{\chain}{\omega}
\newcommand{\tp}{\mathcal{T}}
\renewcommand{\AA}{\mathcal{A}}

\newcommand{\CC}{\mathcal{C}}
\newcommand{\nbc}{\textup{{\bf nbc}}}

\title[Shelling-type orderings and Salvetti complex]{\mbox{Shelling-type orderings of regular CW-complexes} and acyclic matchings of the Salvetti complex}
\author{Emanuele Delucchi}
\email{delucchi@mail.dm.unipi.it}
\keywords{Posets; shellability; Recursive Coatom Orderings; acyclic matchings; Discrete Morse Theory; Oriented Matroids; arrangements of hyperplanes; Salvetti complex; minimal CW-complexes; No Broken Circuit sets}
\address{Dipartimento di Matematica, Universit\`a di Pisa, largo B. Pontecorvo 5, 56127 Pisa, Italia.}

\renewcommand{\1}{\hat{1}}
\newcommand{\0}{\hat{0}}

\begin{document}

\maketitle

\begin{abstract}
Motivated by the work of Salvetti and Settepanella (\cite[Remark 4.5]{SS}) we introduce certain total orderings of the faces of any shellable regular
CW-complex (called shelling-type orderings) that can be used to explicitly construct maximum acyclic
matchings of the poset of cells of the given complex. Building on an
application of this method to the classical zonotope shellings (i.e., those arising from linear extensions of the tope poset) we
describe a class of maximum acyclic matchings for the Salvetti complex
of a linear complexified arrangement. To do this, we introduce and study a new purely combinatorial stratification of the Salvetti complex.
For the obtained acyclic matchings we give an explicit description of
the critical cells that depends only on the chosen linear extension of
the poset of regions. It is always possible to choose the linear
extension so that the critical cells can be explicitly constructed
from the chambers of the arrangement via the bijection to
no-broken-circuit sets defined by Jewell and Orlik \cite{JO}. Our
method generalizes naturally to abstract oriented matroids.
\end{abstract}


\section{Introduction}
The idea of {\em shelling} was initially introduced by Bruggesser and
Mani \cite{mm} as a (geometrically defined) technique to deconstruct
polytopes in a `controlled way' allowing an accurate bookkeeping of
certain combinatorial data. The required total ordering of the
polytope's facets was obtained from the order in which a general position line in meets the affine hulls of the facets. Much work has been spent on a purely combinatorial characterization of this process, and on a corresponding generalization of the method beyond polytopes. In fact, shellability can be defined for general (possibly nonpure) regular cell complexes \cite{BjW1,BjW2}. A line of research initiated by Bj\"orner \cite{Bj} studies combinatorial properties of posets that ensure shellability of the associated order complexes. A considerable amount of work was dedicated to this subject (see e.g. \cite{Bj,BCW,BjW1,BjW2}). 
Particular attention was dedicated to the posets of cells of regular CW-complexes: Bj\"orner characterized them combinatorially (see \cite[Definition 2.1 and Proposition 3.1]{BCW}), and proved that shelling orders of the facets of the associated CW-complex correspond to recursive coatom orderings of the posets (\cite[Proposition 4.2]{BCW}, see also \cite[Theorem 13.2]{BjW2}).

Recently, Forman \cite{For} proposed a combinatorial version of Morse
theory, called {\em Discrete Morse Theory}. The idea is that, given
any regular CW-complex, one can define a combinatorial analog of the
Morse vector fields (i.e., acyclic matchings on the poset of cells;
see Definition \ref{acyclicmatching} and \cite[Proposition 3.3]{Chari}) such that the original
complex is homotopy equivalent to a complex having as many cells of
dimension $d$ as there are `critical points' (i.e., non-matched cells)
of rank $d+1$. Moreover, the attaching maps can be reconstructed from
the knowledge of the `gradient paths' (i.e., alternating paths in the
poset). Since at the topological core of both shellability and
discrete Morse theory lies the idea of collapsing cells (along matched
edges or along the shelling order), it is natural to study the
relation between these concepts: this study was undertaken by various
authors, e.g. in \cite{BaHe,Chari,Koz}.  A comprehensive and careful exposition of the nowadays established combinatorial framework of discrete Morse theory can be found in  the book of Kozlov \cite{Ko}. \\

The motivation for our considerations was given by a joint work of
Mario Salvetti with Simona Settepanella \cite{SS}, where discrete
Morse theory is used to explicitly obtain a minimal CW-complex that
models the homotopy type of the complement of a complexified
arrangement of hyperplanes, thus providing a constructive proof of the
minimality result for general arrangements that was obtained independently by Randell \cite{Rand} and Dimca and Papadima \cite{DiPa}. Another recent study of minimal complexes for complexified arrangements is due to Yoshinaga \cite{Yo}. For the basic definitions about arrangements of hyperplanes we refer to \cite{OT}. 

The starting point of \cite{SS} is the Salvetti complex (introduced in \cite{Sal1} as a combinatorial model for the topology of the complement of complexified arrangements), and the main tool used to construct a
maximum acyclic matching of its poset of cells is a certain total order on the faces of the
arrangement that is called {\em polar ordering} by the authors. The name
refers to the fact that this total order is obtained by considering polar coordinates with respect to a generic flag and
then ordering the faces according to their smallest point in the
lexicographical order of the polar coordinates (for the precise
definition see \cite[Definition 4.4]{SS}).  It is explicitly asked for a completely combinatorial
formulation of this method \cite[Remark 4.5]{SS}.

In an attempt to answer this question, we keep the idea of
constructing acyclic matchings by considering the arrangement from a
`generic' point of view, but we try to stay in the context of oriented
matroids. These are widely studied combinatorial objects that encode
the structure of real arrangements of pseudospheres, and in particular
of linear hyperplanes (for an introductory reference see \cite[Chapter
1]{BLSWZ}). Thus, we actually loose the generality of \cite{SS}, where
the results hold also for affine arrangements. However, our method has
the advantage that it does not need the choice of a generic flag in
the ambient space, and that it holds for general abstract oriented matroids.

One of the ways one can think to look `generically' at an oriented matroid is to consider a shelling of its {\em zonotope}. This is a polytope that is classically associated to every oriented matroid and that, if the oriented matroid corresponds to a real arrangement, is combinatorially isomorphic to the polyhedral subdivision of the unit sphere given by the hyperplanes (for a precise account of this subject, see \cite[Section 2.2 and Chapter 4]{BLSWZ}). It is well-known that to every linear extension of the {\em tope poset} of the oriented matroid corresponds a (class of) shelling(s) of the associated zonotope: in fact, one can construct recursive coatom orderings of the zonotope's poset of faces.

We first show a way to construct maximum acyclic matchings of
\mbox{(CW-)} posets that admit a recursive coatom ordering. We do this
using {\em shelling-type orderings}: a class of total orderings of the
involved poset that are associated to recursive coatom orderings.
Then we apply this construction to the special case of a
zonotope. 

It turns out that linear extensions of tope posets describe also a nice decomposition of the Salvetti complex that, to the best of our knowledge, has not been considered up to now.
The above obtained zonotope shellings give acyclic matchings of every `piece' of this decomposition that can be
`pasted together' to give an acyclic matching of the poset of cells of
the whole Salvetti complex. To every critical cell correspond canonically a
(unique) chamber and a flat of the underlying matroid which
codimension equals the dimension of the critical cell. Both are
uniquely determined by the chosen linear extension of the tope
poset. Maximality of the matching follows from the fact that the
critical cells are in bijection with no-broken circuits, and thus with
generators of the homology (by e.g. \cite{JO,Yo}). 

This correspondence
can be made more precise and explicit: we show that, for an adequate
choice of the ordering of the hyperplanes and of the linear extension
of the base poset, the bijection between chambers and
no-broken-circuits given by Jewell and Orlik in \cite{JO} associates
to every chamber a basis of the flat that carries the corresponding
critical cell.\\

The paper is organized as follows. After introducing the main characters, in Section \ref{comb}  we prove
that every recursive coatom ordering of a CW-poset induces a shelling-type total ordering of its faces (Lemma \ref{p_i}). From this total ordering, in Proposition \ref{cw_ac} we construct an acyclic matching of the given poset that turns out to be `optimal' 
(for a comparison with known
related results of Chari \cite{Chari} and Babson and Hersh \cite{BaHe} see Remark \ref{comm}). Then, Section \ref{oms}
introduces oriented matroids, explains the construction of the
zonotope shelling associated to a linear extension of the tope
poset and compares (in Remark \ref{confronto}) the associated shelling-type ordering with the
polar orderings of \cite{SS}: this is our (kind of)
answer to \cite[Remark 4.5]{SS}. In Section \ref{sal_mat} we study the stratification
of the Salvetti complex induced by a linear extension of the tope
poset (in the context of arrangements also called `poset of regions' and first considered in \cite{ed1}). First, we prove a general property of tope posets (Theorem \ref{propJc}) that, given a linear extension, allows to associate a unique flat $X_C$ to every tope $C$.  It turns out that the stratum associated to a tope $C$
corresponds naturally to the oriented matroid obtained by contraction
of the flat $X_C$. On the one hand, this allows to construct acyclic matchings for every stratum and
to verify acyclicity and maximality of the `patchworked' matching
(Proposition \ref{maxmat}). On
the other hand, in Section \ref{sect_nbc} we show that for some orderings
of the hyperplanes (Definition \ref{cut}) there is a linear
extension of the tope poset (Definition \ref{lex}) for which the flat
$X_C$ is spanned by the no-broken circuit set that corresponds to $C$
under the bijection described in \cite{JO} (Proposition \ref{result}).

\subsection*{Acknowledgments}
The work on which we report was carried out during a stay at the
university of Pisa financially supported by a postdoctoral fellowship
of the Swiss National Science Foundation. Thus, I have to thank
professor Mario Salvetti not only for the inspiring seminars and the useful discussions, but also for the friendly hospitality.

\section{Shellings and acyclic matchings}\label{comb}

\subsection{On partially ordered sets.}

In this work we will deal extensively with partially ordered
structures. We outline the basic definitions, pointing to \cite[Chapter 3]{Sta} for a comprehensive reference. A {\em poset} is a set (say $P$) endowed with a partial order
relation (say $<$), and will be written as a pair $(P,<)$ or, if no
misunderstanding about the partial order will be possible, just denoted by
$P$. Moreover, the posets we will consider will be {\em locally
  finite}, meaning that for each $p\in P$ there are only finitely many
$q$ with $p<q$ or $p>q$. An element $p\in P$ is said to {\em cover}
$q\in P$ whenever $p>q$ and there is no $x\in P$ with $p>x>q$. If $p$
covers $q$ with respect to the order relation $>$ (or $\lexg$, $\tg$,...), then we will write $p\gtrdot q$ (respectively $\lexgdot$, $\tgdot$). The set of all elements of $P$ that are covered by $p$ will be called, by
slight abuse of notations, the set of {\em coatoms} of $p$, and
denoted by $\coat(p)$. In fact, for every $q\in P$, the set $\coat(q)$
is the set of coatoms of the poset $P_{\leq q}:=\{p\in P \mid p\leq q
\}$. This poset is called the {\em principal lower ideal} generated by $q$, a {\em lower ideal} being in general any subposet of $P$ that can be written as an intersection of principal lower ideals; 
 (principal) {\em upper} ideals are defined accordingly. Any subset of the form $P_{\leq q} \cap P_{\geq p}$ is called an {\em interval} of $P$. We
will write $P_{< q}:=P_{\leq q}\setminus \{q\}$. 
A totally ordered subset $\omega\subset P$ will be
called {\em chain}, and its {\em length} is defined by
$\ell(\omega):=\vert\omega\vert -1$. The length of $P$, $\ell(P)$, is then defined as the maximum length of a chain contained in $P$. If every maximal chain of $P$ has the same length, then $P$ is called {\em graded} and possesses a unique {\em rank function} $r:P\rightarrow \mathbb{N}$, $r(p):=\ell(P_{\leq p})$.

A poset $P$ is said to be a {\em lattice} if every two $p,q\in P$ have a unique least upper bound (called {\em join} and denoted $p\vee q$) and a unique greatest lower bound (called {\em meet} and denoted by $p\wedge q$). 
\begin{rem}\label{closedinterval} An upper (lower) ideal in a lattice is principal if and only if it is closed under meet (join).
\end{rem}

Sometimes we will have to consider different order relations on
the same set. If needed, the concerned order relation will be specified in a
subscript. Thus, for example, $\max_{\succ} A$ denotes the maximal element of $A$
with respect to the order $\succ$. A {\em linear extension} of a partial order $<$ is a total order
$\lexl$ such that $p\lexl q$ whenever $p<q$. 

A poset $P$ is called {\em bounded} if it possesses a maximal and a
minimal element. 
Let $\hP$ denote the poset $P$ with
a maximal and a minimal element added, if $P$ has none. The maximal
and minimal elements of $P$ are customarily denoted by $\1$ and, respectively, $\0$. In a poset with $\hat{0}$ a principal lower ideal is also called a {\em lower interval}.

Given a (possibly nonpure) CW-complex $K$, we define its {\em poset of faces} $\FF(K)$ as the set of (closed) cells of $K$ ordered by containment, with a minimal element $\0$ added (the `$-1$ - dimensional cell'). Note that, for every cell $k$, every maximal chain in $\FF(K)_{\leq k}$ has the same length. The {\em height} of $k$ is $h(k)=\ell(\FF(K)_{\leq k})$, the length of the corresponding lower interval. Geometrically, we have $\dim(k)=h(k)+1$ for every cell $k$.

\begin{figure}[h]
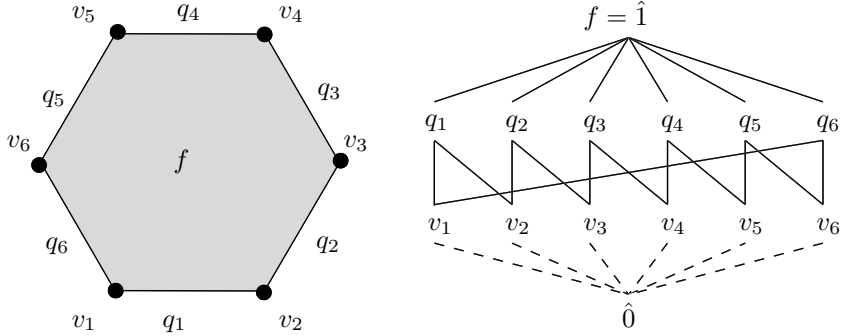

  \begin{picture}(0,0)%
    \includegraphics{hexagon_scaled.pstex}%
  \end{picture}%
  \input{hexagon_scaled.pstex_t}%
  
\caption{The regular CW-complex $K_1$ given by a filled hexagon, and its poset of faces $\HF(K_1)$.}\label{hexg}
\end{figure}

\subsection{Shellability and Recursive Coatom Orderings.}

A regular CW complex $K$ is said to be {\em shellable} if its maximal cells can be given an order along which the complex can be `rebuilt' in a very controlled way. For the precise definition we refer to \cite[Definition 13.1]{BjW2}, where shellability was first extended from simplicial complexes to regular CW-complexes. The complexes that we will consider are given by means of their poset of cells. Therefore we take a point of view that is more tailored to our context: we will define recursive coatom orderings of posets, and then see how they correspond to shellings of regular CW-complexes.

\begin{df}[Definition 5.10 of \cite{BjW1}]\label{def_rca} A bounded poset $(P,<)$ is said to admit a {\em recursive coatom ordering $\rcal$} if $\ell(P)=1$, or if $\ell(P)>1$ and there is a total ordering $\rcal = \rcal_{\1}$ on the set $\coat(\hat{1})$ of coatoms of $P$ such that \begin{itemize}
\item[(1)] for all $p\in\coat (\1)$, the interval $[\0,p]$ admits a recursive coatom ordering $\rcal_p$ in which the coatoms of the intervals $[\0,q]$ for $q\rcal_{\1} p$ come first.
\item[(2)] for all $p\rcal_{\1} q$, if $p,q> y$, then there is $p'\rcal_{\1} q$ and $z\in\coat(q)$ such that $p'>z\geq y$.
\end{itemize}  
\end{df}

This definition is one of the criteria introduced by Bj\"orner to check shellability of the order complex of a poset. It turned out that, in the context of regular CW-complexes, this property is equivalent to shellability. We state these facts in the next theorem.

\begin{thm}[See \cite{Bj},\cite{BjW2}]\label{cwshell}
If a poset $P$ admits a recursive coatom ordering, then $\Delta(P)$ is shellable. If $P$ is the poset of faces of a regular CW-complex $K$, then a total ordering of the maximal faces of $K$ is a shelling order if and only if it is a recursive coatom ordering of $\hP$.
\end{thm}

\begin{figure}[h]
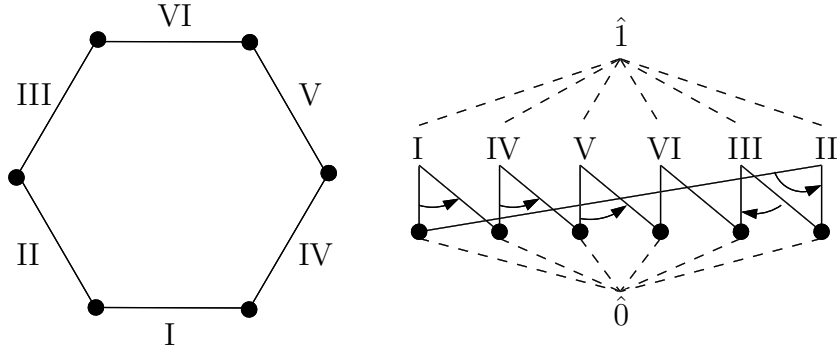

  \begin{picture}(0,0)%
    \includegraphics{hexshell_scaled.pstex}%
  \end{picture}%
  \input{hexshell_scaled.pstex_t}%
  
\caption{A shelling order of the maximal faces of the boundary complex $K_2$ of a hexagon, and a corresponding Recursive Coatom Ordering of the poset $\HF(K_2)$. The arrows give the R.C.O. of the corresponding lower intervals. Below the VI face, the ordering does not matter.}\label{shellingorder}
\end{figure}


\subsection{Matchings and Discrete Morse Theory.}
 
We introduce here some basic concepts of Discrete Morse Theory,
 omitting their proofs. The interested reader will find reference to the
 publications where the statements first appeared. For a
 comprehensive exposition of the subject in its entirety we refer to
 the book of Kozlov \cite{Ko}.

\begin{df}[Compare Proposition 3.3 of \cite{Chari}]\label{acyclicmatching}
Let $(P,<)$ be any poset. The set of {\em edges} of $P$ is $\mathfrak{E}:=\{(p, q)\in P\times P\mid p\gtrdot q\}$. A subset $\M\subset \mathfrak{E}$ is called a {\em matching} of $P$ if every element of $P$ appears in at most one {\em matched pair}, i.e., a pair $(p,q)\in\M$. A {\em cycle} in a matching $\M$ is a subset $\{(p_1,q_1),\ldots , (p_k,q_k)\}\subseteq \M$ such that 
$$q_1\lessdot p_2,\; q_2\lessdot p_3,\; \ldots,\; q_k\lessdot p_1. $$
The matching $\M$ is called {\em acyclic} if it contains no cycle. 

Much of the terminology is borrowed from the theory of graphs, the idea being that $\M$ is actually a matching of the {\em Hasse diagram} of $P$, i.e., the graph defined by the set of edges $\mathfrak{E}$ on the vertex set $P$ (informally speaking, this is the graph one usually draws when graphically representing a poset). A matching $\M$ will be called {\em maximal} if there is no matching $\M'\supsetneq \M$. If, in addition, $\M$ has maximal cardinality among all matchings of $P$, then it is called a {\em maximum matching}. A {\em perfect matching} is a matching such that every element of $P$ is contained in a matched pair. In general, $p\in P$ is called {\em critical} for $\M$ if it is not contained in any matched pair.
\end{df}

The following result is very useful in dealing with acyclic matchings.

\begin{lm}[Theorem 11.2 of \cite{Ko}]\label{ac_linext} A matching $M$ of a poset $P$ is acyclic if and only if there is a linear extension $\lexl$ of $P$ such that $p\lexldot q$ whenever $(p,q)\in M$. 
\end{lm}

From a topological point of view, the interest of acyclic matchings of posets is explained in the following (weak) version of the main theorem of Discrete Morse Theory.
\begin{thm}[Theorem 11.13 of \cite{Ko}. See also \cite{Chari,For}]\label{cwm} Let $K$ be a regular CW-complex $K$ and $\M$ an acyclic matching of $\FF(K)\setminus \{\0\}$. Let $c_i$ denote the number of critical elements of rank $i$.  Then $K$ is homotopy equivalent to a CW-complex that has $c_i$ cells in dimension $i$. 
\end{thm}

\begin{figure}[h]
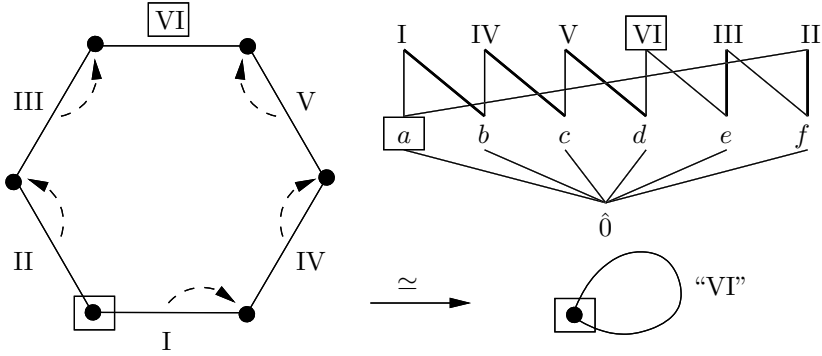

  \begin{picture}(0,0)%
    \includegraphics{acmatch_scaled.pstex}%
  \end{picture}%
  \input{acmatch_scaled.pstex_t}%
  
\caption{To illustrate Theorem \ref{cwm} we take again the empty hexagon $K_2$ of Figure \ref{shellingorder} and its face poset. The bold edges give an acyclic matching which critical cells are the ones in the boxes: one in dimension $1$ and one in dimension $0$. Indeed, the complex is homotopy equivalent to $S^1$. Although we will not get into it here, note that the homotopy equivalence here can be obtained as the concatenation of the collapses indicated by the dashed arrows - that are strongly related with the matching.}\label{acyclic}
\end{figure}

\begin{rem} If we consider the whole $\FF(K)$ we can say that if a
  perfect acyclic matching of $\FF(K)$ exists then $K$ is contractible. Moreover, if there is an acyclic matching of $\FF(K)$ that has critical elements only in one rank level, say the $i$-th, then $K$ is homotopy equivalent to a wedge of $i$-spheres.  
\end{rem}

\subsection{From recursive coatom orderings to acyclic maximum matchings.}
We now describe a construction of certain acyclic maximum matchings of the poset of cells of every shellable regular CW-complex. The core of the argument is Lemma \ref{p_i}, where a convenient linear ordering of all cells is produced.

\begin{rem}\label{comm} It has to be pointed out that our approach via recursive coatom orderings differs from those taken in \cite{BaHe} and \cite{Chari}. Babson and Hersh \cite{BaHe} consider a certain kind of shellability (i.e. lexicographic) of a particular class of simplicial complexes (order complexes of posets) and, in this case, they construct Morse functions ``with a relatively small number of critical cells'' (\cite[Introduction]{BaHe}). Our argument works for any shelling order of any regular CW-complex $K$ and gives always a `best possible' matching. In this sense, when $K$ is the order complex of a poset, and the the shelling order is the lexicographic one, our result improves \cite[Theorem 2.2]{BaHe}. 
After the first version of this paper, we learned that also Chari \cite{Chari} proved the existence of `best possible' matchings for regular CW-complexes with a {\em generalized shelling} (for the precise meaning and the definitions see \cite[Page 103 and Corollary 4.3]{Chari}). Our approach is different, and more constructive. We use the algorithmic language of recursive coatom orderings, and exploit the structure given by the shelling-type linear orderings in the construction of the matching. This structure allows a more accurate understanding of the matchings, and we decided to include it as a stepping stone toward the study of the boundary relations in the minimal complexes produced in Proposition \ref{maxmat}, a task that we plan to undertake in future work.

We would like to point out that our shelling-type orderings appear to be a kind of generalized shellings where the bounded faces are exactly the homology facets of the considered CW-complex.
\end{rem}

As a first step, we define the class of posets that will be the object of our study. It is clear that these posets include the posets of cells of (possibly nonpure) regular CW-complexes.

\begin{df} A poset $P$ will be called {\em locally ranked} if all
  its principal lower ideals are ranked. It then possesses a well-defined
  {\em height function} $h$ that assigns to every element the rank of
  the lower principal ideal it generates. Let $h(P):=\max\{h(p)\mid p\in P\}$. For technical reasons, we will denote by $P_i$ the set of all $p\in P$
  with $h(p)=h(P)-i$.
\end{df}

The set of maximal elements of a given locally ranked poset $P$ will be denoted by $M_P$ or simply $M$ if no misunderstanding can occur. 
If an ordering $\rcal$ of $M_P$ is specified, then we can associate to every $p\in P$ a unique element \begin{equation*}m_p:=\max_{\rcal}\{m\in M_P \mid m\geq p\}\end{equation*} (informally, the last among the maximal elements that lie above $p$).

 \newcommand{\with}{\textrm{ with }}

We proceed to prove the key technical lemma toward Proposition \ref{cw_ac}.

\begin{lm}\label{p_i}
Let $(P,<)$ be a locally ranked poset, and let a recursive coatom ordering $\rcal$ be defined on $\hP$. Then it is possible to define a family of total orders $\{(P_i,\rkl_i)\}_{i=0,\ldots ,h(P)}$ with the following properties:\\ given $p\in P_i$, and writing 
$Q_p:=\bigcup_{p'\rkl_i p} \coat(p')$,\\ 
(1) the order induced by $\rkl_{i+1}$ on $D_p:=\coat(p)\setminus Q_p $ can be extended to a recursive coatom ordering $\rcal_p$ of $\coat(p)$ in which the elements of $Q_p$ come first.\\
(2) for all $p'\rkl_{i}p$ in $P_{i}$, if $p',p> z$, then there is $p''\rkl_{i}p$ and $w\in\coat(p)$ such that $p''>w\geq z$.
\end{lm}

\begin{proof}[{\bf Proof.}] 
The orderings $\rkl_i$ will be defined recursively for increasing $i$. First, since $P_0\subset M$, it makes sense to let $\rkl_0$ coincide with the given recursive coatom ordering $\rcal$. By hypothesis, for every $p\in P_0$ there is a recursive coatom ordering $\rcal_p$ of $P_{\leq p}$ in which the elements of $Q_p$ come first. Therefore we can define $\rkl_1$ by declaring 
$$x\rkl_1 y \Leftrightarrow \left\{\begin{array}{ll}
x\rcal_p y & \textrm{if there is } p\in P_0 \with x,y\in D_p, \\
p\rkl_0 q &  x\in D_p, y\in D_q,\\
m_x\rcal_{\1} y & \textrm{if }  y\in M.
\end{array}\right.$$  
This ordering is well-defined because by construction $D_P\cap D_q =\emptyset$ if $p\neq q$. Moreover, it clearly satisfies the requirement.

Now let $i>1$ and suppose that the orderings $\rkl_j$ are defined for $j\leq i$. \\
{\em Definitions:} For $p\in P_i$ let $q_p:=\min_{\rkl_{i-1}}\{q\in P_{i-1}\mid q>p\}$ (and note that this implies $p\in D_q$). 
Moreover, let $$N_p:=\coat(p)\setminus \{y\in\coat(p')\mid p'\in\coat(q_p),\, p'\rkl_j p\}$$ and note that by definition  $P_{i+1}=\coprod_{p\in P_i}N_p$.
We define also $$A_p:=\{y\in \coat(p)\mid y<q' \textrm{ for } q'\rkl_{i-1} q_p \}\textrm{ and } B_p:=Q_p\cap P_{<q_p},$$
so that $N_p=\coat(p)\setminus B_p$ (see Figure \ref{fig1}).

\begin{figure}[h]
  \begin{picture}(0,0)%
    \includegraphics{morsefig_scaled_2.pstex}%
  \end{picture}%
  \input{morsefig_scaled_2.pstex_t}%
  
\caption{}\label{fig1}
\end{figure}

{\em Remark:} For every $p\in P_i$ we have $A_p\subseteq B_p$. In fact,
given  $p\in P_i$ and $x\in A_p$, by assumption on $\rkl_{i-1}$ there is $w\in \coat(q_p)$ such that $w>x$ and $w\rkl_i p$.

Because $\rkl_i$ induces a recursive coatom ordering on $\coat(q_p)$, we know that there is a recursive coatom ordering $\rcal_p$ of $\coat(p)$ such that the elements of $B_p$ come first. 

For $x,y\in P_{j+1}$ we define:
$$x\rkl_{i+1} y \Leftrightarrow \left\{\begin{array}{ll}
x\rcal_p y & \textrm{if there is } p \textrm{ such that } x,y \in N_p,\\
p\rkl_i p' & \textrm{if } x\in N_{p},\, y\in N_{p'},\\
m_x\rcal y & \textrm{if } y\in M 
\end{array}\right.$$  

At this point it is worth to point out that, given $p\in P_i$, $Q_p=\bigcup_{p'\rkl_i p} N_p$ and $D_p=N_p$.

We have now to check the conditions. (1) is clear: given $p\in P_i$ and
 $D_p=N_p$, $\rcal_p$ is a recursive coatom ordering of $\coat(p)$ such that the
 elements of $B_q$, and thus every $x\in \coat(p)\setminus
 Q_p$, come first. For (2) take $x,x'\in P_{i+1}$ such that
 $x'\rkl_{i+1}x$ and  $z<x',x$. If $x'\in N_{p'}$ and $x\in N_p$ for
 $p\neq p'$, then we have $p'\rkl_i p$, and by property (2) of
 $\rkl_i$ there is $p''\rkl_i p$ and $y\in \coat(p)$ such that $z\leq
 y\leq p''$. Applying Definition \ref{def_rca}.(2)  to
 $\rcal_p$ we obtain an $x''\rcal_p x$ and a $w\in \coat(x)$ such
 that $z\leq w < x''$. The proof is concluded by the remark that
 $x''\rcal_p x$ implies $x'' \rkl_{i+1} x$ because $x\in N_p$. 
\end{proof}

\begin{df}[Shelling-type orderings] Let $P$ be a locally ranked poset.
We introduce functions $\pi_i: P_i\rightarrow P_{i+1}$ defined by
$$\pi_i(q):=\max_{\rkl_{i+1}}\{p\in P_{i+1}\mid q>p \}, $$
where the $\rkl_i$ are the orderings associated to some shelling via Lemma \ref{p_i}.
 
Then we define a linear extension $\lexl$ of $P$ by:
$$p\lexl q \Leftrightarrow \left\{\begin{array}{ll}
p\rkl_i q & \textrm{ if there is }i \textrm{ such that }p,q\in P_i,\\
p\rkleq_{i}\pi_{i}\pi_{i+1}\cdots \pi_{j-1}(q) & \textrm{ if }p\in P_i,\, q\in P_j\textrm{ and }i>j.
\end{array}\right.$$ 
The easy check that this is a well-defined linear order is left to the
reader. Every linear extension $\lexl$ of $P$ 
that is constructed in this way from a recursive coatom ordering 
will be called a {\em shelling-type ordering} of $P$.
\end{df}

We can now construct an acyclic matching for any shelling-type ordering of a locally ranked poset.

\begin{lm}\label{rco_matching} Every shelling-type ordering $\lexl$ 
of a locally ranked poset $P$ induces an acyclic matching $\M$ on $P$. 
\end{lm}

\begin{proof}[{\bf Proof.}] For $i=1,\ldots , h(P)$ let $\rkl_i$ denote the restriction of $\lexl$ to $P_i$. By definition, every $(P_i,\rkl_i)$ satisfies the claim of Lemma \ref{p_i}. We write  $P_i=\{p^i_1,\ldots , p^i_{k_i}\}$, where $p^i_j \rkl_i p^i_{j+1}$ for all $j=1,\ldots , k_i$.


{\em Definition of the matching} $\M$: 

 
\noindent We start with the one-element matching $\M:=\{(p^1_1, \pi_1(p^1_1))\}$. For every $j=2,\ldots , k_1$ we add $(p^1_j, \pi_1(p^1_j))$ to $\M$ if $\pi_1(p^1_j)$ is not already matched (or, equivalently, if $\pi_1(p^1_l)\neq \pi_1(p^1_j)$ for all $l<j$).\\
For $i=1,\ldots h(P)$ we further expand $\M$ as follows:
for $j=1,\ldots , k_i$, if $p^i_j$ is not already matched and $\pi_i(p^i_j)\neq \pi_i(p^i_l)$ for all $l<j$, then add $(p^i_j,\pi_i(p^i_j))$ to $\M$.  \\
Since, by construction, $p\lexgdot \pi_i(p)$ whenever $(p,\pi_i(p))\in\M$, this matching is acyclic by Lemma \ref{ac_linext}.
\end{proof}

So far we stayed in the full generality of locally ranked posets. If we restrict ourselves to the case of posets of cells of CW-complexes, we can have even more control on the critical elements. The stepping stone for this is the following easy lemma, that we prove for completeness.

\begin{lm}\label{sphere} Let $K$ be a regular CW-decomposition
  of a sphere. Then in every shelling order of $K$ the only
  homology facet is the last one.
\end{lm}
\begin{proof}[{\bf Proof.}]
The argument is by contraposition. Indeed, if the claim would not
hold, then there would be a counterexample, say a regular CW-complex
$K$, a homeomorphism $\phi: K \rightarrow S^d$, and a shelling order on
the facets of $K$ such
that the last facet, call it $F$, is not a homology facet. 
This means that the union $K'$ of all the facets other than $F$ is a
shellable complex with still a homology facet - in particular, it is
not contractible. But on the other hand, this complex has to be
homeomorphic to $S^d\setminus \phi(F\setminus K')$, which is
contractible because $F\setminus K'$ is. A
contradiction follows.
\end{proof}

%
\begin{prop}\label{cw_ac} Every shelling of a regular CW-complex $K$ induces an acyclic matching of the poset of faces of $K$. Moreover, the critical cells of this matching correspond to the homology facets of the given shelling. 
\end{prop}

\begin{proof}[{\bf Proof.}]  It is known that every shelling of a regular CW-complex corresponds to a recursive coatom ordering  of its poset of cells (see e.g. \cite[Theorem 13.2]{BjW2}) and, by Lemma \ref{p_i}, to a family of orderings $(P_i,\rkl_i)$ giving rise to a shelling-type ordering $\lexl$. Via our Lemma \ref{rco_matching}, this $\lexl$ every shelling order of the facets of $K$ defines an acyclic matching $\M$ of the poset of cells $P:=\FF(K)$. We have to study the critical cells.

Consider a critical element $p\in P_i$ such that $p$ is not maximal in $P$. Several situations can occur:\\
{\em (i) There is $q\gtrdot p$ such that $(q,\pi_{i-1}(q)) \in \M$.} Then $p\rkl_i\pi_{i-1}(q)$ and, since $p$ was not matched, there must be $\tilde{p}\rkl_i p$ such that $\pi_i(p)=\pi_i(\tilde{p})$. In particular, every element of $x\in\coat(p)$ is coatom of some $p'\rkl_i p$ by \ref{p_i}.(1). We may assume without loss of generality that $p'\in\coat(q)$, because else by property (2) of Lemma \ref{p_i} we can find $p''\in \coat(q)$ such that $x<p''$. This all means that, in the shelling of $P_{<q}$ that is induced by $\rkl_i$, the whole boundary of $p$ is already taken when the turn of $p$ comes. But since $p$ is not the last element of this shelling (which is $\pi_{i-1}(q)$), using Lemma \ref{sphere} we get a contradiction with the fact that $P_{<q}$ is a shellable sphere. This case can therefore not enter.$\diamondsuit$\\
{\em (ii) There is $q\gtrdot p$ that is not matched.} If for this $q$ we have $p\rkl_i \pi_{i-1}(q)$, then the same reasoning of item (i) applies to get a contradiction. On the other hand, if $\pi_{i-1}(q)=p$ then our algorithm should have taken the edge $(q,\pi_{i-1}(q)=p)$ into $\M$ when examining the elements of $P_{i-1}$: indeed, $p$ was not already taken as $\pi_{i-1}(q')$ for any $q'\rkl_{i-1} q$ (and actually it will remain free until the end!). So, this second situation can also not happen.$\diamondsuit$\\
{\em (iii) Else:} every $q\gtrdot p$ is matched `from above', i.e., by an edge $(w,\pi_{i-2}(w)=q)$. In this case, let $q_1,\ldots, q_k$ be any enumeration of the elements that cover $p$. We know that no edge $(q_j,p)$ is matched, but we have supposed also that for every $j=1,\ldots,k$ there is $w_j$ such that $(w_j,q_j)\in \M$. Since $P$ is a CW-poset, we know (e.g. by \cite[Proposition 2.2]{BCW}) that every interval of length $2$ has four elements - so that to every $j\in\{1,\ldots,k\}$ we can associate $\phi(j)\in\{1,\ldots , k\}$ such that the interval $[p,w_j]$ has elements $\{p, q_j, q_{\phi(j)}, w_j\}$. In this interval by assumption the edge $(w_j,q_j)$ is matched, and therefore for sure $(w_j,q_{\phi(j)})\not \in \M$. This implies in particular $w_j\neq w_{\phi(j)}$ for every $j$. But then the alternating path $q_j,w_j,q_{\phi(j)},w_{\phi(j)},q_{\phi^2(j)},\ldots$ must be a cycle, because $\phi$ can take only finite many values (we supposed the CW-complexes to be locally finite). Thus, also this case cannot enter.$\diamondsuit$

It follows that every critical element is a maximal element of $P$, i.e., by a facet of $K$. But a maximal element $m\in P_i$ is not matched exactly when $\max_{\rkl_{i+1}}\coat(m)$ is matched by some $p\rkl_i m$ (and hence, by item (i) above, when all its coatoms are). In topological words, $m$ is critical exactly if, when its turn in the shelling comes, its whole boundary was already taken. This means exactly that $m$ is a homology facet of the given shelling.     
\end{proof}

\begin{ex}\label{sto} The acyclic matching depicted in Figure \ref{acyclic} is induced from the shelling order and the recurdive coatom ordering of Figure \ref{shellingorder} by the following shelling-type ordering: $$I\lexl a \lexl b \lexl II \lexl f \lexl III \lexl e \lexl IV \lexl c \lexl V \lexl d \lexl VI \lexl \hat{0}.$$
\end{ex}

\begin{rem}
Proposition \ref{cw_ac} gives a perfect acyclic matching of $\hP$ whenever $P$ is the poset of faces of a regular CW-complex that is homotopy equivalent to a sphere. Indeed, in that case the only critical cell of $P$ can be matched by the added element $\1=\hP\setminus P$.
\end{rem}


\section{Shelling-type orderings of oriented matroids}\label{oms}

In this section we apply Proposition \ref{cw_ac} to a special situation, as an attempt to answer \cite[Remark 4.5]{SS} and as a stepping stone to the results of Section \ref{sal_mat}. If we consider the fan defined by a set of real linear hyperplanes, we see that the boundary of the associated polar polytope is a shellable (CW-) sphere. The combinatorics of  real arrangements of hyperplanes is customarily encoded by {\em oriented matroids}. 
These combinatorial objects are more general than real linear hyperplane arrangements; however, to every oriented matroid corresponds a shellable CW-sphere that, in case the oriented matroid describes an arrangement, is combinatorially isomorphic to the associated polar polytope. 

In what follows we state the precise definitions and recall the
results that we will need for this paper. The standard reference for a
comprehensive overview on oriented matroids is \cite{BLSWZ}.

In Remark \ref{confronto} we will return to the case where the oriented matroid comes from an arrangement of real hyperplanes to compare our shelling-type orderings to the polar ordering of \cite{SS}.

\newcommand{\VV}{\mathcal{V}}

\begin{df}[Oriented matroid]\label{OM}  Given a ground set $E$, a collection $\VV\subset \{+,-,0\}^E$ is the set of {\em vectors of an oriented matroid $\MM$} if and only if following properties are satisfied:

\begin{itemize}

\item[(1)] $(0,0,\ldots , 0)\in \VV$,

\item[(2)] if $X\in \VV$, then $-X\in \VV$,

\item[(3)] for all $X,Y\in \VV$, $X\circ Y \in \VV$,

\item[(4)] for all $X,Y\in \VV$,\\ given $e,f\in E$ such that $X_e=-Y_e$ and not both $X_f,Y_f$ equal $0$, \\there is $Z\in\VV$ such that $Z_e=0$, $Z_f\neq 0$, and if $Z_i\neq 0$ then $Z_i$ equals $X_i$ or $Y_i$.

\end{itemize}\end{df}

Let us point out that this is only one of the many ways to characterize oriented matroids. For a complete account of the many different possible axiomatizations we refer to Chapter 5 of \cite{BLSWZ}.

\begin{rem} Let $\VV$ be the set of vectors of an oriented matroid $\MM$. Let $\VV^*$ denote the set of all $G\in \{+,-,0\}^E$ such that $\sum_{e\in E} G_eX_e = 0$ for all $X\in\VV$ (the multiplication and the sum being performed by thinking of $+$ as $+1$ and of $-$ as $-1$). Then $\VV^*$ is the set of {\em covectors} of $\MM$. It is a matter of fact that $\VV^*$ satisfies the axioms of Definition \ref{OM}: it is the set of vectors of an oriented matroid that is called {\em dual} to $\MM$ (note that $(\VV^*)^*=\VV$). For a proof of this see \cite[Proposition 3.7.12]{BLSWZ}.
\end{rem}

The {\em support} of a subset $X\subset \{+,-,0\}^E$ is $\supp(X):=\{e\in E\mid X_e\neq 0\}$.  We define a partial order on $\VV$ by setting 

$$X\leq Y \Leftrightarrow \forall e\in E: \left\{\begin{array}{l}

\supp(X)\subseteq\supp(Y) \textrm{ and} \\

Y_e\neq 0 \Rightarrow X_e\neq -Y_e.

\end{array}\right.$$

\begin{df}\label{faces_OM}The set $\VV^*$ endowed with this ordering is called the {\em face poset} of the oriented matroid $\MM$ and is denoted by $\FF(\MM)$. It has a unique minimal element but in general it possesses many maximal elements, that are called {\em topes} of the oriented matroid. the set of topes of an oriented matroid $\MM$ will be denoted by $\tp(\MM)$ (or just $\tp$).

For $T\in\tp$ and $F\in \VV^*$ we define $T_F\in\tp$ by $(T_F)_e=T_e$ if $F_e= 0$ and $(T_F)_e=F_e$ else (see Remark \ref{geomint} for a geometric interpretation of this operation).
\end{df}

 It turns out that also the set $\tp$ can be given  interesting partial orders. These were introduced by Edelman \cite{ed1} in the context of arrangements of hyperplanes and independently by Edmonds and Mandel \cite{EM} for abstract oriented matroids.

\begin{df}[See also Definition 4.2.9 of \cite{BLSWZ}]
Let an oriented matroid $\MM$ be given and consider its set of topes $\tp$. For $T,T'\in\tp$ let $$S(T,T'):=\{e\in E\mid T_e=-T_e'\}.$$
To every tope $B\in\tp$ we can associate a partial order $\tl_B$ on $\tp$ defined by
$$T_1\tl_B T_2 \Leftrightarrow S(B,T_1)\subset S(B,T_2).$$

The set $\tp$ endowed with the order relation $\tl_B$ is called {\em tope poset of $\MM$ based at $B$} and will be denoted by $\tp_B(\MM)$ or simply by $\tp_B$. This poset is ranked by $r(T)=\vert S(B,T) \vert$. We will use the symbol $\texl$ to indicate total orderings that are linear extensions of the ordering of a tope poset.
\end{df}

It is a nice fact that, for any oriented matroid $\MM$, 
$\FF(\MM)^{op}$ (suitably augmented by an additional $\hat{0}$-element, if needed)
is the poset of faces of a convex polytope that is called the {\em zonotope} of $\MM$. The $1$-skeleton of its dual polytope is isomorphic, as a graph, to the Hasse diagram of $\tp_B(\MM)$ for every $B\in \tp$. In this sense, specifying a linear extension of $\tp_B$ amounts to somehow `specify a direction' in the ambient space of the zonotope. Indeed, such a linear ordering is all what one needs to get a shelling of the zonotope.

\begin{figure}[h]
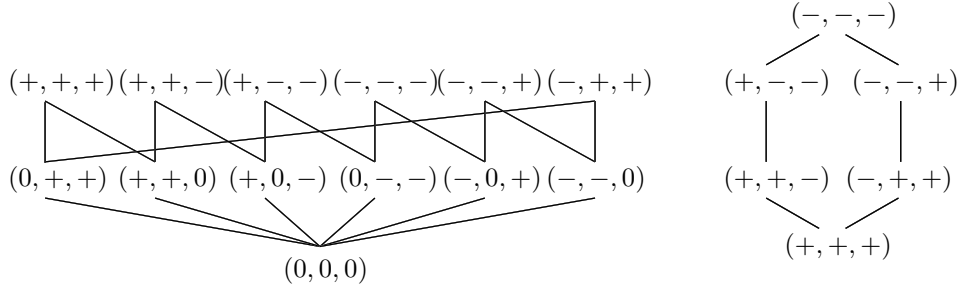
\begin{center}
\noindent%
  \begin{picture}(0,0)%
    \includegraphics{covectors_scaled.pstex}%
  \end{picture}%
  \input{covectors_scaled.pstex_t}%
  \end{center}
\caption{On the left is the face poset $\FF(\MM)$ of an oriented matroid on $3$ elements. Its (augmented) dual $\FF(\MM)^{op}$ appeared already in Figure \ref{hexg}, so that the zonotope of this oriented matroid is the hexagon $K_1$. The dual polytope of $K_1$ is again an hexagon, so that the tope poset $\tp_{(+,+,+)}(\MM)$ for this oriented matroid is the poset depicted on the right. }\label{covectors}
\end{figure}

\begin{thm}[Proposition 4.3.2 of \cite{BLSWZ}]
Let $\MM$ be a simple oriented matroid and $B$ be a tope of $\MM$.
Every linear extension of the tope poset $\tp_B(\MM)$ induces a recursive coatom ordering of $\FF(\MM)$.
\end{thm}

Thus, an application of Proposition \ref{cw_ac} gives immediately the
following existence result.

\begin{thm}\label{linext_acmatch}
Let $\MM$ be a simple oriented matroid and $B$ be a tope of $\MM$.
Every linear extension $\texl$ of the tope poset $\tp_B(\MM)$ defines an acyclic matching $\M$ of the face poset $\FF(\MM)$ such that the only critical element is $-B$, the tope opposite to $B$.
\end{thm}

\begin{ex}\label{exlinex} One possible linear extension of the tope poset of Figure \ref{covectors} is given by
$$ (+,+,+) \texl (+,+,-) \texl (+,-,-) \texl (-,+,+) \texl (-,+,+) \texl (-,-,+) \texl (-,-,-). $$

Comparing Figure \ref{shellingorder} we see that this linear extension corresponds indeed to the shelling order $I,II,\ldots, VI$ of $K_2$ via the correspondence of the posets of faces, and thus induces on the poset $\FF(\MM)=\FF(K_2)$ the acyclic matching indicated in Figure \ref{acyclic}. 
\end{ex}

\begin{rem}[On polar orderings]\label{confronto}
As we will explain in detail in the next Section, to every real linear arrangement of hyperplanes is associated an oriented matroid whose covectors correspond to the induced stratification of $\mathbb{R}^n$. These `special' oriented matroids can be therefore also given a {\em polar ordering} in the sense of Salvetti and Settepanella \cite{SS}. This makes a comparison of the two orderings possible. The outcome is that shelling-type orderings are
different from the polar orderings of \cite{SS} on linear arrangements (although they can be used for the same scope, as we will see in the next section): indeed, a polar
ordering is {\em never} a linear extension of the face poset (as can
be easily seen comparing Theorem 4 of \cite{SS}). Moreover, the order
induced on the chambers by a polar ordering is never a linear
extension of a poset of regions: otherwise, there would be no other choice for the base chamber $B$ as to
take the chamber containing the basepoint of the polar ordering. But
then we see that there is a maximal chain in $\tp_B$ (determined by the general
position line $V_1$ of \cite{SS})  whose elements form by definition an
initial segment in the order of chambers induced by the polar
ordering. This is clearly incompatible with being a linear extension
of $\tp_B$.

Nevertheless, at a first glance the ordering induced on the chambers
by the polar orders seems to be a shelling order for the zonotope. We
leave this as an open question.
\end{rem}

\begin{rem}\label{rem_0} The proofs of \cite[Proposition 4.3.1 and 4.3.2]{BLSWZ} are
  constructive. Hence, by taking a closer look at the arguments used
  there one can give an explicit description of the shelling-type
  orderings (and thus of the matchings) that result from our construction. To
  do this, we need some notation. For every element $e$ of the oriented matroid let $R(e):=\min_{\texl}\{B_F\mid |F|=e\}$, and let $F(e)$ be the unique face with $F(e)\lessdot R(e)$ and $|F(e)|=e$. Then, for every $R\in\tp_B$ choose a maximal chain $\chain_R$ in the interval $[B,-R]\subset \tp_{R}(\MM)$. For $i=0,\ldots, n$ let $\omega_R(i)$ denote the $i$-th element of the chain (counted from the bottom). 

For every maximal element $R$ of $\F$
we can express $D_R$, $Q_R$ and $\pi_0(R)$ (see Lemma \ref{p_i}) as follows: \begin{center}$D_R=\{F\in \coat(R)\mid |F|\in S(-R)\}\quad(=\{F\in \coat(R)\mid R_F=T_F\})$,\\ $Q_R=\{F \in \coat(R) \mid |F|\in S(R)\}\quad(=\{F\in \coat(R)\mid R_F\neq T_F\})$,\\ $\pi_0(R)=\omega_R(n)$.\end{center} We conclude that the induced ordering $\rkl_1$ on $\F$ can be expressed by
$$F_1\rkl_1 F_2\Leftrightarrow\left\{\begin{array}{l}
T_{F_1}\texl T_{F_2}\textrm{ or }\\
T_{F_1}= T_{F_2}=:R \textrm{ and } |F_1|\Tl_{R} |F_2|,
\end{array}\right.$$ 
where $\Tl_R$ is the order in which the elements appear as $S(\omega_R(i),\omega_R(i+1))$ for increasing $i$.

Moreover, according to the construction of \cite[4.3.1 and 4.3.2]{BLSWZ}, the recursive
ordering of $\coat(F)$ is given as above by any maximal chain in $\tp_F(\MM/|F|)$ that contains $F'$, where $F'$ is the face where $\chain_R$ crosses $|F|$. In particular, the elements of $D_F$ are ordered according to a maximal chain in the interval $[F',-F]\subset \tp_F(\MM/|F|)$, and so on.
\end{rem}



\section{Acyclic maximum matchings for the Salvetti complex}\label{sal_mat}

\newcommand{\sless}{<_s}

The main motivation of Salvetti and Settepanella for considering polar
orderings in \cite{SS} was to use these total orderings in the
construction of what they call the {\em polar gradient}. The polar
gradient of \cite{SS} is essentially an acyclic maximum matching of the poset
of cells of the Salvetti complex - a regular CW complex that was
introduced by Mario Salvetti in order to model
the homotopy type of the complement of a complexified arrangement of hyperplanes (see Definition \ref{sal_cplx} and \cite{Sal1}).\\

In this section we want to construct acyclic matchings for the Salvetti
complex of linear arrangements using shelling-type orderings. In fact,
the outcome is that linear extensions of tope posets give a very nice
stratification of the Salvetti complex (see Lemma \ref{contr}) and
allow us
to paste together different choices of acyclic matchings of the
strata.\\

Let us begin by the definition of the poset of cells of the Salvetti
complex for a general oriented matroid. We present it here in his
general form and as a formally defined object to underline the fact
that it can be defined in purely combinatorial terms. In a second step
we will introduce the terminology (and the geometric intuition) of
arrangements of hyperplanes.



\begin{df}\label{sal_cplx}
Given an oriented matroid $\MM$, we define a poset $\SS(\MM)$ (denoted
simply by $\SS$ if no confusion can arise). The elements of $\SS(\MM)$ 
are all pairs $\langle F, T
\rangle$ where $F\in\FF(\MM)$, $T\in \tp(\MM)$ and $F<T$ in
$\FF(\MM)$. The order relation in $\SS$ will be denoted $\sless$
and defined by setting $$\langle F,T\rangle \sless \langle F',T'\rangle
\textrm{ if } F>F' \textrm{ in } \FF(\MM) \textrm{ and } T=T_F'.$$
Recall that the poset $\FF(\MM)$ has a unique minimal element that we
denote by $P$. For any
given tope $T$ let $\SS_T:=\SS(\MM)_{\leq \langle P,T\rangle}$. It is
clear that $\SS_T$ is isomorphic to $\FF(\MM)^{op}$ as a poset. If no confusion can arise we will write just $\SS$, $\FF$, $\tp$ for $\SS(\MM)$, $\FF(\MM)$, $\tp(\MM)$.

Now fix a  ``base tope'' $B\in \tp$. If a linear extension $\texl$ of $\tp_B$ is given, define, for every $R\in \tp$, $$\SS(R):=\bigcup_{T\texl R} \SS_T \quad \textrm{ and  } \quad N(R):= \SS(R)\setminus \SS(R'),$$ where $R'$ is the tope that precedes $R$ in $\texl$.

\end{df}

\begin{ex} The poset of Figure \ref{salposf} is $\SS(\MM)$ for the (realizable) oriented matroid $\MM$ of Figure \ref{covectors}, where for better readability we denoted the covectors by the corresponding strata in $\mathbb{R}^2$ (see Figure \ref{arrgt}). 
\end{ex}

 A {\em real arrangement of hyperplanes} is a set $\AA:=\{H_1,\ldots H_n\}$ where the $H_i$ are codimension $1$ affine subspaces of $\mathbb{R}^d$. The arrangement $\AA$ is called {\em linear} if every $H_i$ is a linear subspace. The combinatorial data of a real linear arrangement $\AA$ is encoded by the associated oriented matroid $\MM_\AA$ of the signed linear dependencies among the vectors $\{v_1,\ldots, v_n\}$ where, for all $i$, $v_i$ is normal to $H_i$. An oriented matroid that is of the form $\MM_\AA$ for some real linear arrangement $\AA$ is called {\em realizable}.

\begin{figure}[h]
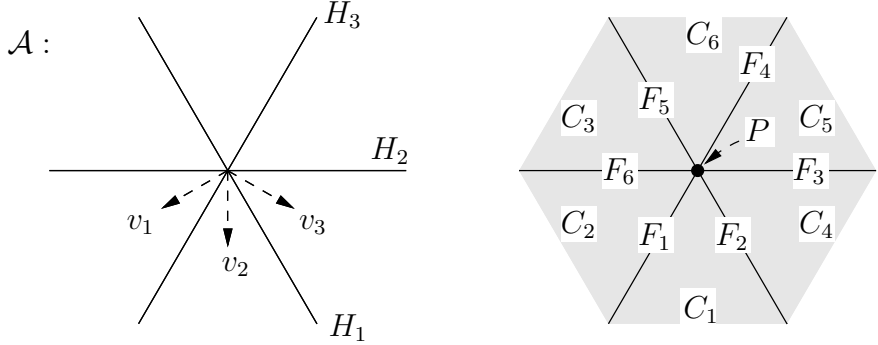

  \begin{picture}(0,0)%
    \includegraphics{om_strata_scaled.pstex}%
  \end{picture}%
  \input{om_strata_scaled.pstex_t}%
  
\caption{Our main example, the arrangement of three lines in the plane. On the left the `plain' arrangement, with our choice of normal vectors to build the oriented matroid $\MM_{\AA}$. On the right, the cells of the induced stratification of $\mathbb{R}^2$.}\label{arrgt}
\end{figure}


The relevance of Definition \ref{sal_cplx} comes from the following fundamental result by Mario Salvetti (which actually holds also in a general form for affine arrangements).

\begin{thm}[Theorem 1 of \cite{Sal1}]
Let $\AA$ an arrangement of linear real hyperplanes. Then $\SS(\MM_\AA)\cup\{\0\}$ is the poset of cells of a regular CW-complex, called {\bf Salvetti complex}, that is homotopy equivalent to the complement in $\mathbb{C}^d$ of the complexification of $\AA$. 
\end{thm}
\newcommand{\HH}{\mathcal{H}}
\newcommand{\WW}{\mathcal{W}}

We see that, although $\SS$ can be defined for any oriented matroid,
the main topological interest of the construction is in the
context of arrangements of hyperplanes. Therefore we will from now
sometimes use the more geometrically intuitive language of this setting, that we are going to explain.

If $\MM$ is a { realizable} oriented matroid corresponding to the arrangement $\AA$, then the poset $\FF(\MM)$ is the poset of the
closed strata determined by $\AA$ in $\mathbb{R}^d$, ordered by 
inclusion of the topological closures.

\begin{ex} For $\AA$ as in Figure \ref{arrgt}, $\MM_{\AA}$ is the oriented matroid $\MM$ of Figure \ref{covectors}. In particular, we can compare the poset $\FF(\MM)$ of Figure \ref{covectors} with the stratification of $\mathbb{R}^2$ on the right hand side of Figure \ref{arrgt}. For instance, the covector $(+,0,-)$ represents the stratum of all vectors of $\mathbb{R}^2$ which scalar product with $v_1$ is positive, with $v_2$ equals $0$ and with $v_3$ is negative (i.e., the points `in front of' $H_1$, `on' $H_2$ and `behind' $H_3$ with respect to the base chamber $B=(+,+,+)$).  
\end{ex}

The topes are the maximal
strata - i.e., the closure of the connected components of the complement
$\mathbb{R}^d\setminus \bigcup\AA$ of $\MM_\AA$ - and are customarily
called {\em chambers} (or {\em regions}) of $\AA$ (given a set
 $A:=\{a_1,a_2,\ldots, a_n\}$, we
will write $\bigcap A$ for the set $a_1\cap a_2\cap\ldots\cap a_n$ and $\bigcup A$ for  $a_1\cup a_2\cup\ldots\cup a_n$). Accordingly, $\tp_B(\AA)$ is often referred to as the {\em poset of regions} of $\AA$ (e.g., in his first appearance in the context of hyperplane arrangements, see \cite{ed1}). For any two chambers $C_1,C_2$ of $\AA$ (topes of $\MM_\AA$), the elements of $S(C_1,C_2)$ correspond to the hyperplanes that {\em separate}\footnote{The use of the word `separation' arose in the litarature while considering the chambers to be the {\em open} sets that are obtained subtracting $\AA$ from $\mathbb{R}^d$, so that any two chambers are really disjoint and `separated' by the hyperplanes in the set $S(C_1,C_2)$. For consistency we let here the chambers be, as any other face, closed. The combinatorics of course works as well, and we will save some cumbersome distinctions in the last section.} $C_1$ from $C_2$, i.e., the hyperplanes that are met by any line segment connecting a point in the interior of $C_2$ with a point in the interior of $C_1$.  Since the arrangements corresponding to oriented matroids are linear, every chamber is a convex cone. The hyperplanes supporting the facets of the cone determined by the chamber $C$ are called {\em walls} of $C$. The set of walls of $C$ is denoted by $\WW_C$. 
\begin{rem}\label{adjacent} For every wall $H\in\WW_C$ there is a chamber $K\in\tp(\AA)$ such that $S(C,K)=\{H\}$. In fact, this can be taken as the `abstract' definition in the setting of arbitrary oriented matroids.
\end{rem}
\begin{nt}\label{not1} We will denote by $\LL(\AA)=\LL(\MM)$ (or just by $\LL$) the
{\em  lattice of flats} of the underlying matroid; this is indeed a
geometric lattice and we will think of it as of the
poset of intersections of the hyperplanes ordered by reverse
inclusion (see the top of Figure \ref{figurone} for a picture of
$\LL(\AA)$ when $\AA$ is the arrangement of three lines through the
origin of the plane). For every face $F\in\FF(\MM)$ we write $\vert F\vert$ for
what corresponds to the ``affine span'' of $F$, i.e., the flat given by the elements of
$\supp(F)$.  Given any flat $Y\in\LL$, we
denote by $\AA_Y$ the arrangement given by the hyperplanes that
contain $Y$ and set $\AA_Y=:\supp(Y)$. We write $\AA^Y$ for the arrangement $\{H\cap Y \mid
H\not\in \AA_Y\}$ that is determined on $Y$ by the hyperplanes that
intersect $Y$ nontrivially. The oriented matroid associated to $\AA^Y$
is the {\em contraction} $\MM(\AA)/Y$ of the oriented matroid
associated to
$\AA$ (see \cite[Section 3.3]{BLSWZ}).   The natural map $\tp(\AA)\rightarrow \tp(\AA_Y)$ will be
denoted by $\pi_Y$. We will use it to explain the geometric
content of the operation described in Definition \ref{faces_OM}.\end{nt}

\begin{rem}\label{geomint}
Let $\MM$ be a realizable oriented matroid and $\AA$ the corresponding arrangement. Let $C$ be one of its topes (chambers) and
$F$ be some covector (face) of $\MM$ ($\AA$). Then the tope $T_F$
corresponds to the unique chamber that is contained in $\pi_{\vert
  F\vert}(T)$ and contains $F$.
\end{rem}


\begin{imprem}\label{ir} In all what follows, unless explicitly stated,\\[0.1cm]
\begin{minipage}{\textwidth}\begin{center} 
\fbox{{\em $\AA$ will denote a finite arrangement of $n$ linear hyperplanes in $\mathbb{R}^d$.}}\end{center}
\end{minipage}
Moreover, we fix from now an (arbitrarily chosen) base chamber $B\in\tp(\AA)$ and a (also arbitrary) linear extension $\texl$ of $\tp_B(\AA)$.  
\end{imprem}

Let us also point out that everything we will say can be easily translated into the language of (and thus: holds for) abstract oriented matroids. As the `grammar' and the `vocabulary' for this translation we refer to \cite{BLSWZ}.

\newcommand{\JJ}{\mathcal{J}}

\begin{nt}\label{psi} Given $H\in \AA$, let
  $\AA':=\AA\setminus\{H\}$. Given $C\in\tp(\AA)$, we will write $C'$
  for the unique chamber of $\AA'$ that contains $C$.  This natural
  inclusion of chambers 
induces an order preserving map
$$\psi: \tp_{B'}(\AA')\rightarrow \tp_B(\AA);\quad\quad C'\mapsto \min_{\texl}\{ C\in\tp_B(\AA)\mid C\subseteq C'\}.$$
Note that if $C'\in\tp(\AA')$ contains two chambers $C_1,C_2\in\tp(\AA)$ then, up to renumbering, $C_1\tldot_B C_2$. So this definition could have been phrased as well in terms of $\tleq$, the partial ordering of $\tp_B(\AA)$, instead of $\texl$.

This map is clearly injective, and thus for $C'_1,C'_2\in \tp(\AA')$ the ordering $\tleq'$ of $\tp_{B'}(\AA')$ satisfies
\begin{equation}\label{trick}
C'_1\tleq' C'_2 \Leftrightarrow \psi(C'_1)\tleq \psi(C'_2).
\end{equation}

Given any linear extension $\texl$ of $\tp_B(\AA)$ we let let $\texl'$
denote the linear extension of $\tp_{B'}(\AA')$ that is in a sense the
`pullback' of $\texl$ along $\psi$:
$$C'_1\texl' C'_2 :\Leftrightarrow \psi(C'_1)\texl\psi(C'_2).$$

\end{nt} 

As we will see this construction is canonical.
\begin{lm}\label{quadratocomm} Given  two distinct hyperplanes $H_1, H_2 \in \AA$, for
  both $i=1,2$ write $\AA_i:=\AA\setminus\{H_i\}$ and let $B_i$ be the unique chamber of $\AA_i$ containing $B$. Let $\psi_i$ denote the map $\tp_{B_i}(\AA_i)\rightarrow \tp_B(\AA)$ defined in \ref{psi}. Let then $\widehat{B}$ be the unique chamber of $\AA_1\cap\AA_2$ that contains $B_1$ and $B_2$, and write $\widehat{\psi}_i$ for the corresponding map $\tp_{\hat{B}}\rightarrow \tp_{B_i}(\AA_i)$. Then the diagram of poset maps
$$\begin{CD} \tp_{\widehat{B}}(\AA_1\cap\AA_2) @>\widehat{\psi}_1>> \tp_{B_1}(\AA_1)\\
 @V\widehat{\psi}_2VV @V{\psi}_1VV \\
\tp_{B_2}(\AA_2) @>{\psi}_2>> \tp_{\hat{B}}(\AA)
\end{CD}    $$
commutes.
\label{quadrato}
\end{lm}

\begin{proof} For brevity, let $\widehat{\AA}:=\AA_1\cap\AA_2$. Consider $\widehat{C}\in\tp(\widehat{\AA})$. By definition we have
$$\widehat{\psi}_i(\widehat{C}) = \min_{\tleq_i}\{C'\in\tp_{B_i}(\AA_i)\mid C'\subset \widehat{C}\}, $$
where $\tleq_i$ is the ordering of $\tp_{B_i}(\AA_i)$. This, in view of equation \ref{trick}, means
$$\min_{\tleq}\{C\in\tp(\AA)\mid C\subset \widehat{\psi}_i(\widehat{C})\}\tleq \min_{\tleq} \{C\in\tp(\AA)\mid C\subseteq C'\subseteq\widehat{C}\textrm{ for a }C'\in\tp_{B_i}(\AA_i) \}$$
or, equivalently,
$$\min_{\tleq}\{C\in\tp(\AA)\mid C\subset \widehat{\psi}_i(\widehat{C})\}\tleq \min_{\tleq} \{C\in\tp(\AA)\mid C\subseteq\widehat{C} \}. $$
Now, because we are taking away from $\AA$ exactly two hyperplanes, the right side of the last expression takes the minimum over a poset that either has only one element, or is a two-element chain, or has four elements and rank two (depending on whether none, one or both of $H_1$ and $H_2$ cut $\widehat{C}$). Thus, in any case the right side above identifies a unique $C\in\tp_B(\AA)$, and this is $\psi_i\widehat{\psi}_i(\widehat{C})$. Summarizing, we have
$$\psi_i\widehat{\psi}_i(\widehat{C})=\min_{\tleq} \{C\in\tp(\AA)\mid C\subseteq \widehat{C}\}.$$ 
Since this expression does not depend on $i$, we are done.
\end{proof}

We will need the following corollary.

\begin{crl}\label{indind} In the setting of Lemma \ref{quadrato}, for $i=1,2$ let $\texl_i$ be the linear extension induced from $\texl$ on $\tp_{B_i}(\AA_i)$, and $\widehat{\texl_i}$ the linear extension of $\tp_{\widehat{B}}(\widehat{A})$ induced from $\texl_i$. Then
$$\textrm{for all }\widehat{C},\widehat{K}\in\tp_{\widehat{B}}(\AA_1\cap\AA_2),\quad\quad\widehat{C}\; \widehat{\texl}_1\; \widehat{K}\;\; \Leftrightarrow\;\; \widehat{C}\; \widehat{\texl}_2\; \widehat{K}$$
\end{crl}
\begin{proof} For both $i=1,2$, $\widehat{C}\; \widehat{\texl}_i\; \widehat{K} $ if and only if $\psi_i\widehat{\psi}_i(\widehat{C})\texl \psi_i\widehat{\psi}_i(\widehat{K})$. The claim follows with Lemma \ref{quadratocomm}. 
\end{proof}
 
Now we can define the the object we will study in the next few
statements. Recall that we fixed a linear extension $\texl$ of the
tope poset of $\AA$.

\begin{df}\label{ilungo} 
For every $C\in\tp(\AA)$ we let 
$$\JJ(C):=\{X\in\LL(\AA)\mid \supp(X)\cap S(C,K)\neq \emptyset\textrm{ for every }K\texl C\},$$
which is easily seen to be an upper ideal in $\LL(\AA)$.\end{df}



\begin{nt} Let $H\in\AA$  be given and recall the notation \ref{psi}. We write $\JJ'(C')$ for the order ideal of $\LL(\AA')$ associated to $C'$ and $\texl'$ in Definition \ref{ilungo}. The inclusion $\AA'\hookrightarrow \AA$ induces an order preserving injection $$\iota: \LL(\AA')\rightarrow \LL(\AA),\, X\mapsto \bigcap \supp(X).$$ We will identify $\JJ'(C')$ with its image under this map.
\end{nt} 


\begin{lm}\label{perxc} 
Let a chamber $C\in\tp_B(\AA)_{<\hat{1}}$ be given, choose $H\in \AA\setminus S(B,C)$ (such an hyperplane exists because $C\neq -B$) and let $\AA':=\AA\setminus\{H\}$.

For every $Y\in\JJ(C)$ 
we have
$$\bigcap (\supp(Y)\setminus\{H\})\in\JJ'(C').$$
\end{lm}

\begin{proof}[{\bf Proof.}] Fix any $Y\in\LL(\AA)$. 
As a first step, observe that\begin{center}
{\em ($\star$) If $H\not\in\supp(Y)$, then $Y\in\JJ(C)\Leftrightarrow
  \bigcap(\supp(Y)\setminus\{H\})\in\JJ'(C')$},\end{center} because in
this case $\bigcap(\supp(Y)\setminus \{H\})=Y$, and the conditions for being in $\JJ'(C')$ and $\JJ(C)$ become equivalent. Therefore suppose from now $H\in\supp(Y)$. 

We want to argue by induction on $\vert\AA\vert$. If $\vert\AA\vert=1$ the claim is trivial. So suppose $\vert\AA\vert >1$ and that the claim holds for every smaller arrangement. We need to distinguish two cases:

{\em \underline{Case 1:} $Y=\hat{1}\in\LL(\AA)$.} In this situation
$$\bigcap(\supp(Y)\setminus\{H\})=\hat{1}\in\LL(\AA'). $$
Since both $\JJ(C)$ and $\JJ'(C')$ are nonempty upper ideals, we have $Y\in\JJ(C)$ and $\bigcap(\supp(Y)\setminus\{H\})\in\JJ'(C')$ and the claim holds.

{\em \underline{Case 2:} $Y\neq \hat{1}\in\LL(\AA)$.} Thus we can find
$\widetilde{H}\in\AA\setminus\supp(Y)$. Since $H\in\supp(Y)$, in
particular $\widetilde{H}\neq H$. We need a couple of definitions, in
order to apply Lemma \ref{quadrato}. 

Let $\widetilde{\AA}:=\AA\setminus\{\widetilde{H}\}$,
$\widetilde{\texl}$ the induced linear extension,
$\widetilde{\JJ}(\widetilde{C})$ the corresponding upper ideal (where
$\widetilde{C}$ is the unique chamber containing $C$) and define
$\widetilde{Y}:=\bigcap(\supp(Y)\setminus\{\widetilde{H}\})$. 
Moreover, let $\widetilde{A}':=\widetilde{\AA}\setminus\{H\}$ ($=\widetilde{\AA}\cap\AA'$) and
define $\widetilde{\texl}'$, $\widetilde{C}$ and
$\widetilde{\JJ}'(\widetilde{C}')$ (noting that by Corollary
\ref{indind} it does not matter to specify whether
$\widetilde{\texl}'$ is induced by $\widetilde{\texl}$ or $\texl'$). We have the following implications:\begin{itemize}
\item[(I)] $Y\in\JJ(C)\Rightarrow \widetilde{Y}\in \widetilde{\JJ}(\widetilde{C})$, e.g. by ($\star$).
\item[(II)] $ \widetilde{Y}\in \widetilde{\JJ}(\widetilde{C}) \Rightarrow
\bigcap(\supp(\widetilde{Y})\setminus\{\widetilde{H}\})\in\widetilde{\JJ}'(\widetilde{C}')
$ by the inductive hypothesis, since $H\in S(C,-B)\subseteq
S(\widetilde{C},\widetilde{B})$ and
$\vert\widetilde{\AA}\vert<\vert\AA\vert$ (here $\widetilde{\texl}'$ is viewed
as being induced from $\widetilde{\texl}$).
\item[(III)] $\bigcap(\supp(\widetilde{Y})\setminus\{\widetilde{H}\})\in\widetilde{\JJ}'(\widetilde{C}')\Rightarrow
\bigcap(\supp(Y)\setminus\{H\})\in\JJ'(C')$ again by ($\star$), where
we used Corollary \ref{indind} in switching point of view and considering
$\widetilde{\texl}'$ to be induced from $\texl'$.
\end{itemize}
The lemma follows by chaining up these implications.
 \end{proof}

\newcommand{\B}{\mathcal{B}}
\newcommand{\D}{\mathcal{D}}
\newcommand{\Bu}{\B^{\uparrow}}
\newcommand{\Bd}{\B^{\downarrow}}
\newcommand{\U}{\mathcal{U}}



\begin{thm}\label{propJc}
For every $C\in\tp_B(\AA)$,  $\JJ(C)\subset\LL(\AA)$ is a principal upper ideal.
\end{thm}

\begin{proof}[{\bf Proof.}]  We will again argue by induction on the size of $\AA$, for if $\AA$ contains only one hyperplane the claim is trivial. So suppose $\vert\AA\vert > 1$, and let the claim hold for every arrangement of size at most $\vert\AA\vert -1$. Choose chambers $B,C\in\tp(\AA)$ and a linear extension $\texl$ of $\tp_B(\AA)$. We will prove that the associated $\JJ(C)$ is closed under the join operation (see Remark \ref{closedinterval}). 

If $C=-B$, then clearly $\JJ(C)=\{\hat{1}\}\subset\LL(\AA)$ and the claim holds. If $C$ is not $-B$, in particular there is $H\in S(C,-B)=\AA\setminus S(B,C)$, and $\AA':=\AA\setminus\{H\}$ satisfies the theorem by induction hypothesis. 

By Lemma \ref{perxc} the (order preserving) map
$$\lambda: \LL(\AA)\rightarrow \LL(\AA'), Y\mapsto \bigcap (\supp(Y)\setminus\{H\}$$  
satisfies $\lambda(\JJ(C))\subseteq \JJ'(C')$. Note that the inclusion $\iota$ of $\JJ'(C')$ into $\JJ(C)$  is well defined because whenever $K\texl C$, then $K'\texl C'$ and $S(C',K')\cap\supp(Y)\subset S(C,K)\cap\supp(Y)$: if the former is nonempty, then so is the latter. 

If we look at the composition of $\lambda$ with $\iota$, we see that $\iota\lambda(Y)\leq Y$ in $\LL(\AA)$ for every $Y\in\JJ(C)$. Now consider two elements $Y_1,Y_2\in\JJ(C)$: by induction hypothesis $\lambda(Y_1)\wedge\lambda(Y_2)$ exists in $\JJ'(C')$. In $\JJ(C)$ we then have an element $\iota(\lambda(Y_1)\wedge\lambda(Y_2))\leq Y_1\wedge Y_2$. Since $\JJ(C)$ is an upper ideal in the lattice $\LL(\AA)$, the proof is complete.
\end{proof}

This theorem ensures the existence of the object that we are going to define. For a {\em construction} of this object one needs some more refined considerations that we will carry out in Section \ref{sect_nbc}.

\begin{df}\label{defX}
Choose, as usual, a base chamber  $B\in\tp(\AA)$, let a linear extension $\texl$ of $\tp_B(\AA)$ be given, and recall Definition \ref{ilungo}.

For every $C\in\tp(\AA)$ define
$$X_C:=\min\JJ(C).$$

\end{df}

From the arguments stated above we can also obtain

\begin{crl}\label{ciliegina} With the assumptions and notations of
  Definition \ref{defX}: $$\textrm{if we define }F_C:=X_C\cap C,\textrm{ we have }|F_C|=X_C.$$
\end{crl}

\begin{proof}[{\bf Proof.}]  Let $\AA$, $B$ and $\texl$ be given, and
  consider $C\in\tp(\AA)$. We will show that $ \dim(X_C\cap
  C)=\dim(X_C)$ whenever $C\neq -B$ (in the remaining case, there is nothing to show). 

Since the claim is trivial when $\vert\AA\vert =1$, we will proceed by induction, assuming from now that $\vert\AA\vert >1$ and that the claim holds for every arrangement with at most $\vert\AA\vert - 1$ hyperplanes.

 Choose $H\in\WW_C\cap S(C,-B)$ (this can be done whitout loss of generality) and note that then $C$ is the intersection of $C'$ with the (closed) halfspace $H^+$ bounded by $H$ and containing $B$. Thus, $$C=C'\cap H^+.$$ We will write $X_C=\min\JJ(C)$ and  $X'_{C'}:=\min\JJ'(C')$. By induction hypothesis: $$ \dim(X'_{C'}\cap C')=\dim(X'_{C'}).$$

Recall now the maps defined in the proof of Theorem \ref{propJc}. We have $$\lambda(X_C)= X'_{C'}$$
 by injectivity of $\iota$. 

\noindent Therefore, only two cases can happen: either $\bigcap\supp(X_C)=\bigcap\big(\supp(X_C)\setminus\{H\}\big)$, and thus $X_C=X'_{C'}$, or $\bigcap\supp(X_C)\neq\bigcap\big(\supp(X_C)\setminus\{H\}\big)$, which implies $X_C=X'_{C'}\cap H$.

If $X_C=X'_{C'}$, 
then in particular $X_C\subset H$ and thus
 $$\dim(C\cap X_C) = \dim(C'\cap H^+\cap X'_{C'}) = \dim(X'_{C'}\cap H^+)=\dim(X_{C}).$$

If on the contrary $X_C=X'_{C'}\cap H$, then 
 $$\dim(C\cap X_C) = \dim(C'\cap H^+ \cap X'_{C'}\cap H)$$$$=\dim(C'\cap X'_{C'}\cap H)=\dim(X'_{C'}\cap H)=\dim(X_C)$$

\end{proof}

\begin{qu} It seems likely that the previous arguments can be carried out also for arrangements of affine hyperplanes, at least if $B$ is assumed to be an {\em unbounded} chamber. Since this is not directly relevant for this work, we leave this as a question.
\end{qu}

We return to the `linear' case. The following lemma states, for later reference, an easy reformulation of the definition of $X_C$.

\begin{lm}\label{tec_lm} By Definitions \ref{ilungo} and \ref{defX}, the flat $X_C$ is uniquely determined by the following properties:\begin{itemize}
\item[(1)] $S(K,C)\cap\supp(X_C)\neq\emptyset$ for all $K\texl C$, and

\item[(2)] For every  $Y\in\LL(\AA)$ such that $Y\not > X_C$ there is a chamber $K\texl C$ such
  that  $S(K,C)\cap\supp(Y)=\emptyset$.
\end{itemize}
\end{lm}

\begin{proof}[{\bf Proof.}]  
Clear.
\end{proof}

The next lemma shows the point of the above definitions: the $X_C$ actually describe in very compact way the strata $N(C)$ of Definition \ref{sal_cplx}.

\begin{lm}\label{contr} Let $\MM$ denote the oriented matroid associated to a real, linear arrangement $\AA$, choose a base region $B\in\tp(\AA)$ and a linear extension $\texl$ of $\tp_B(\AA)$, and recall Definition \ref{sal_cplx}. Then $$N(C)\simeq \F(\MM/X_C).$$\end{lm}
\begin{proof}[{\bf Proof.}]  By definition $N(C)=\{\langle F,C \rangle \in \SS_C\mid C_F\neq K_F \textrm{ for all } K\texl C \}$. Since the order is induced by $\SS_C$, we only have to prove equality of sets.

The right-to-left inclusion is easy. Indeed, if
$F\in\F(\MM/X_C)$, then $S(C_F,K)\cap \supp(F)=S(C,K)\cap \supp(F)$
for all $K$. By Lemma \ref{tec_lm}.(1), for all $K\texl C$ we have $S(C,K)\cap\supp(F)\neq \emptyset$, and thus $C_F\neq K_F$ . For the other direction, suppose $\langle F; C
\rangle \in N(C)\setminus \F(\MM/X_C)$,

so that  $F< F'$ in $\FF^{op}$, hence $|F'|< X_C$. Then by Lemma \ref{tec_lm}.(2) there is
$K\texl C$ with $ S(K,C)\cap \supp(F)=\emptyset$, and thus $K_F=C_F$:
a contradiction.\end{proof}

Now we can apply the preceding work to construct a family of maximum
acyclic matchings of the Salvetti complex. 

\begin{figure}
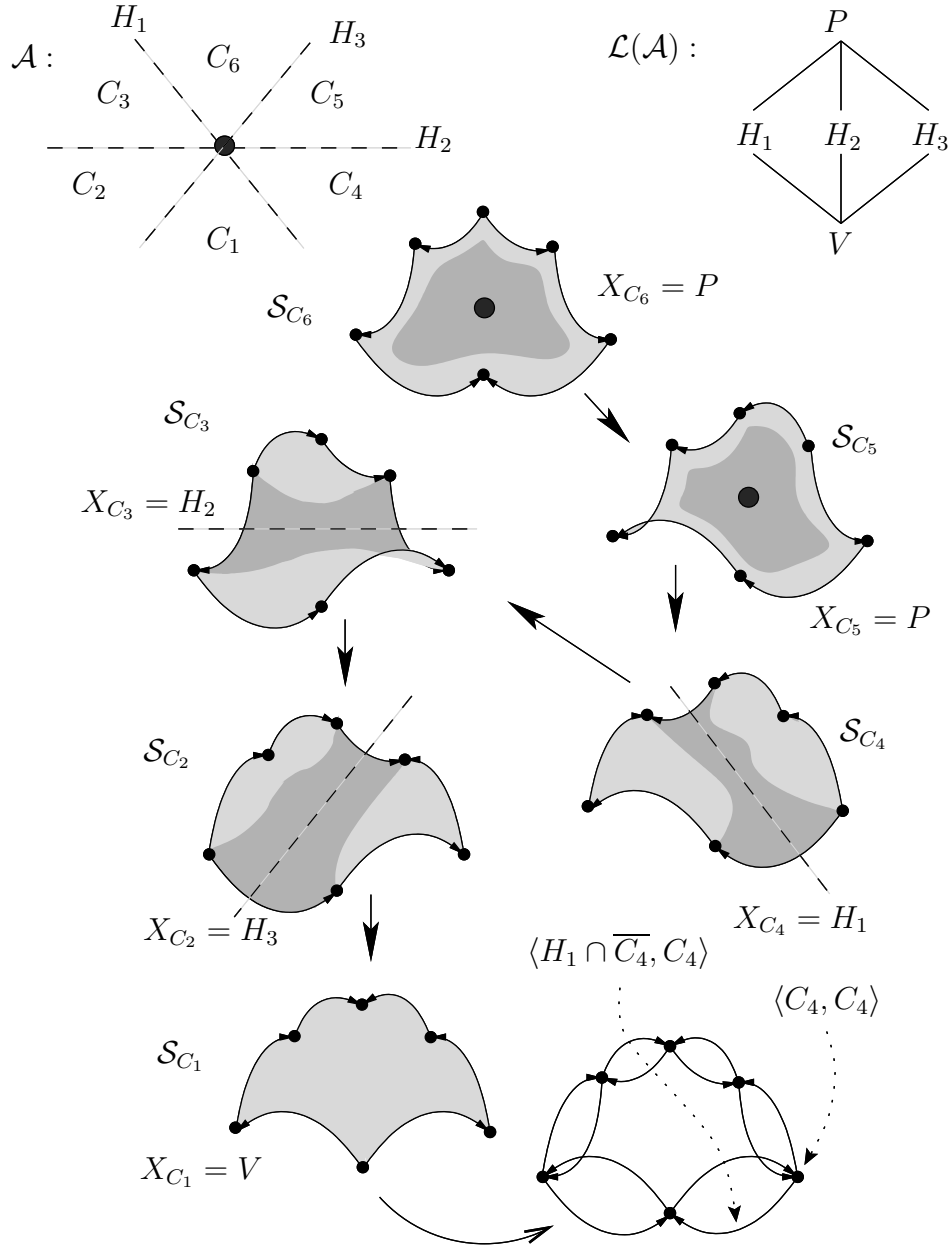

  \begin{picture}(0,0)%
    \includegraphics{strata_scaled.pstex}%
  \end{picture}%
  \input{strata_scaled.pstex_t}%
  
\caption{The Salvetti complex for the arrangement of three lines in the plane, ``assembled'' by attaching the top cells to the $1$-skeleton along the linear extension of the tope poset that was described in Example \ref{exlinex} (see also Figure \ref{covectors} and \ref{arrgt}) The shaded regions represent the `contributions to homotopy' that every top cell gives to the total complex.}\label{figurone}
\end{figure}

\begin{prop}\label{maxmat}
Let $\AA$ be an arrangement of linear hyperplanes in real space and fix any $B\in\tp(\AA)$. 
To every linear extension of $\tp_B(\AA)$ corresponds a family of acyclic maximum matchings of the associated Salvetti complex $\SS(\MM_\AA)$ which critical cells are in natural bijection with the chambers of $\AA$.
\end{prop}
\begin{figure}[b]
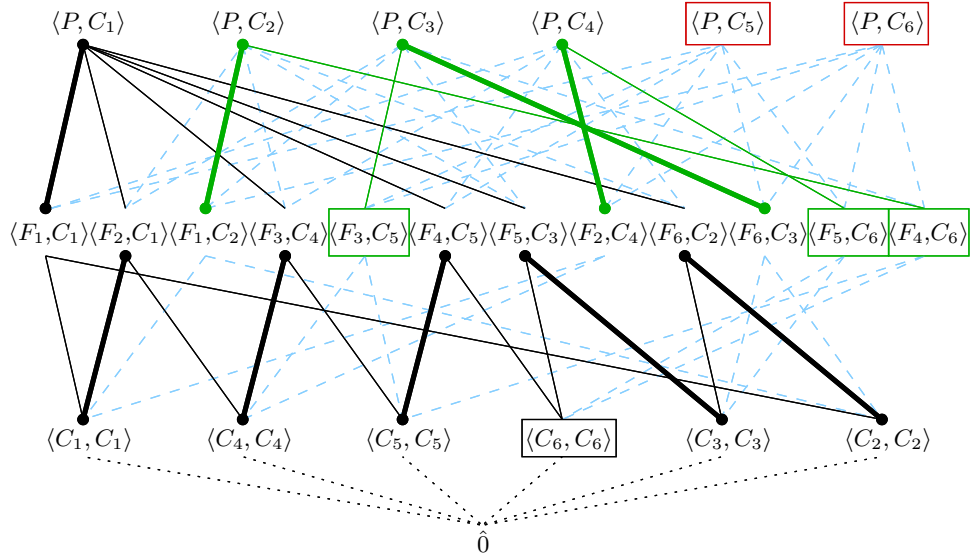

  \begin{picture}(0,0)%
    \includegraphics{sal_c_m_b_scaled.pstex}%
  \end{picture}%
  \input{sal_c_m_b_scaled.pstex_t}%
  
\caption{The poset of cells of the Salvetti complex for the
  arrangement of Figure \ref{arrgt}, where the chambers were numbered according to our chosen linear extension of the tope poset (see Example \ref{exlinex}).
The dashed lines relate elements in different strata; elements of the same stratum are joined by solid lines. The stratum $N(C_1)$ is drawn in black, the strata $N(C_2),N(C_3),N(C_4)$ are drawn in green, while for $i=5,6$ we have $N(C_i)=\langle
  P,C_i\rangle$. The stratification corresponds to the one of Figure
  \ref{figurone}.
 Note that the induced shelling-type ordering of Example \ref{sto} translates into: $C_1 \lexl F_1 \lexl F_2 \lexl C_2 \lexl F_6 \lexl C_3 \lexl F_5 \lexl C_4  \lexl F_4 \lexl C_5 \lexl F_3 \lexl C_6 \lexl P$.
 On each stratum we depict the associated acyclic matching
 by thickening the edges of the matching. The resulting
  critical cells are enclosed into boxes.}\label{salposf}
\end{figure}

\begin{proof}[{\bf Proof.}] 
Let $\texl$ denote a linear extension of the ordering $\tl_B$ of $\tp_B$ and recall Definition \ref{sal_cplx}.

We will prove recursively that every poset $\SS(C)$ possesses a
maximum acyclic matching with as many critical cells as there are
chambers $C' \texl C$.

For $\SS(B)$ this follows from Theorem \ref{linext_acmatch}; so let the claim hold for a
chamber $C\texg B$.  We have to find an acyclic matching of the `new' part $N(C)$.

For any chamber $K$ let
$$
N(C,K):=\SS_C \setminus \SS_K 
=\{\langle F,C\rangle\in\SS_C\mid C_F\neq K_F\}.
$$
Clearly $N(C)=\bigcap_{K\texl C} N(C,K)$, and thus, with every
$N(C,K)$, also $N(C)$ is an upper ideal in $\SS(C)$. Since by Lemma \ref{contr} $N(C)$ is the face poset of an oriented matroid, with Theorem \ref{linext_acmatch} we have an
acyclic matching of $N(C)$. These matchings can be pasted together to
give a matching of the whole $\SS$.   The acyclicity of the
'patchwork-matching' can be shown with Lemma \ref{ac_linext} by
considering the linear extension of $\SS$ given by the concatenation of
the linear extensions of the $N(C)$s so that an element of $N(C_1)$
comes after an element of $N(C_2)$ whenever $C_1\texl C_2$ (for a
precise proof see the more general statement of \cite[Theorem
11.10]{Ko} on `patchwork of acyclic matchings'). 

 By Theorem \ref{linext_acmatch} , the shelling induced on $N(C)$ has only one homology cell, and thus the corresponding acyclic matching has exactly one critical element.
With the `pigeon hole principle' we now see that the
obtained `global' acyclic matchings on $\SS$ are in fact maximum acyclic matchings:
indeed, the number of critical elements and the number of generators
in homology both equal the cardinality of the family of the {\em no
  broken circuit sets} (see e.g. \cite{JO}).

\end{proof}

\begin{rem}\label{grezzo} The matchings of the previous proposition are obtained
  by pasting together acyclic matchings for the different
  $N(C)s$. In principle, any choices of acyclic maximum matchings of the $N(C)$s
  can be pasted together. But since it is easy to see that a
  shelling-type ordering of a locally ranked poset restricts to a
  shelling-type ordering of any of its lower ideals, we can
  construct the whole matching keeping the freedom of choice to a
  minimum: it is possible to give an explicit description of the critical elements of the matching induced on $\SS$ by the choice of a base chamber $B$, of a linear extension $\texl$ of $\tp_B$, and of maximal chains  $\chain_C$ in $[B,-C]$ for all $C\in \tp$:
the critical point added with $N(C)$ is $\langle F(C),C_{F(C)}\rangle$, with $$F(C):=\max{}_{\rkl_{r(C)}}\{F'\in \FF\mid |F'|=X_C \},$$
where 
$\rkl$  is the shelling-type ordering induced on $\FF^{op}$ and $r(C)$ is the rank (i.e., the codimension) of $X_C$.
\end{rem}

\section{No broken circuits and critical elements}\label{sect_nbc}

In this last section we want to relate our construction to
no-broken-circuit sets. It is not easy to track back the origin of these widely studied combinatorial objects that can be
defined for every geometric lattice; let us here mention just \cite{Bry,Bjnbc} as `early references'. We only
recall that they give a basis for the Whitney homology of the
associated geometric lattice (see \cite{Bacl,Bjnbc}) and, in the context of arrangements of
hyperplanes, the no-broken-circuit sets of size $k$ index a
basis of the $k$-th degree of the Orlik-Solomon algebra (see e.g. \cite{OS,JL} and the textbook \cite{OT}), which is
known to be isomorphic to the (integral) cohomology algebra of the
arrangement's complement \cite{OS}. For a
comprehensive and very readable account of these objects, and for more bibliography, see the
survey of Yuzvinsky \cite{YuzSurv}.

We will continue our `geometric' treatment of the subject and, as
above, leave to the interested reader the translation into the
language (and the strength) of abstract oriented matroids.

\begin{df}(no-broken-circuit sets)\label{defnbc}  Translating the classical definition for matroids, a {\em circuit} of $\AA$ is a minimal set $\CC$ of hyperplanes such that every $H\in
\CC$ contains the intersection of the other elements of $\CC$. In
particular, for every $H\in \CC$ the set $\CC\setminus\{H\}$ is
minimal with the property that the intersection of its hyperplanes
equals $\bigcap \CC$. If a linear ordering of the set of hyperplanes is given, a {\em broken circuit} is a subset $B\subset \AA$ that can be written as $\CC\setminus \{H\}$, where $H$ is the minimal element of $\CC$ in the chosen total order.

A {\em no-broken-circuit set}, also called simply {\em nbc set}, is an
independent subset of $\AA$ that contains no broken circuit, or the
empty set. It is clear that the nbc sets give a simplicial
complex, denoted $\nbc(\AA)$, on the ground set $\AA$. Note that we
formally consider also the simplex of dimension $-1$ given by the
empty set - thus, $\emptyset \in \nbc(\AA)$ for all $\AA$. 
\end{df}

\begin{ex} For the arrangement $\AA$ of three lines in the plane, with the
  lattice depicted on the top right of Figure \ref{figurone}, we have only
  one circuit, namely $\{H_1,H_2,H_3\}$, and thus we get
$$\nbc(\AA)=\big\{\emptyset,\{H_1\},\{H_2\},\{H_3\},\{H_1,H_2\},\{H_1,H_3\}\big\}.$$
\end{ex}

A corresponding notion exists for arbitrary geometric lattices (i.e., for arbitrary matroids): the interested reader is referred to \cite{Bjnbc}. 

It is important to point out that, for technical reasons, our definitions differ from those of \cite{JO} in that our broken circuits fail to contain a {\em minimal} (instead of a maximal) element. The other definitions are then adapted to this change.

Before to state the main definitions, let us fix some notation that will accompany us through the remainder of this paper.

\begin{nt}\label{not2} We keep the conventions of the Important Remark \ref{ir} but now, in addition, we suppose a linear ordering $\{H_1,\ldots ,H_n\}$ to be given on the set of hyperplanes. For the moment no special requirements are made on this ordering.

We will write
$$\AA_j:=\{H_1,\ldots,H_j\}\textrm{ for }1\leq j\leq n,\quad\quad\AA':=\AA_{n-1},\quad\quad \AA'':=\AA^{H_n},$$
where $\AA^{H_n}=\{H\cap H_n\mid H\in\AA'\}$, according to the Notation \ref{not1}. 
Clearly every $\AA_j$ inherits the ordering from $\AA$. Moreover, there is a canonical ordering of $\AA^{H_n}$ obtained by numbering every element $L\in\AA''$ according to the `smallest' hyperplane $H(L)\in\AA$ in which it is contained.
As above, every $C\in\tp(\AA)$ is contained in exactly one chamber of $\AA'$, that we will denote by $C'$. Thus, $B'$ is the only chamber of $\AA'$ that contains the base chamber $B$ of $\AA$.

For every $H\in\AA$ let $H^+$ denote the closed halfspace
that is bounded by $H$ and contains $B$. Clearly
$B=\bigcap_{H\in\AA}H^+$ and $B'=\bigcap_{H\in\AA'}H^+$. More generally, there
is a canonical choice of a base region $B_j$ for $\AA_j$: we define
$B_j:=\bigcap_{i\leq j} H^+_i$. Turning our attention to $\AA''$, for
$L\in\AA''$ it is natural to define $L^+:=H_n\cap H(L)^+$. 
Now,  {\em if $H_n$ is a wall of $B$} write $B'':=\bigcap_{L\in\AA''}L^+$. 
\end{nt}

The last requirement on $H_n$ is necessary to ensure that the
intersection defining $B''$ has indeed maximal dimension inside $H_n$. It is clear that with this
hypothesis $$B''=B'\cap H_n.$$ We will
need this property to hold inductively: this is the motivation of the
following definition.

\begin{df}[Cut property]\label{cut}  A total ordering $\{H_1,\ldots,H_n\}$ of $\AA$
  satisfies the {\em cut property} with respect to the base chamber
  $B$ if, for every $j=2,\ldots,n$, $H_j$
  intersects the interior of $B_{j-1}$ (we will say: $H_n$ {\em cuts} $B_{j-1}$).
\end{df}

We need to check that an ordering with this property
exists. The next Lemma explains that those orderings correspond to known
objects. Namely: maximal chains in the poset of regions.

\begin{lm} An ordering $\{H_1,\ldots,H_n\}$ of the hyperplanes of an
  arrangement $\AA$ satisfies the cut property if and only if there is
  a maximal chain $$B=C_0\tl C_1\tl\ldots \tl C_n=-B$$ in $\tp_B(\AA)$
  such that $S(C_{i-1},C_i)=\{H_i\}$ for all $1\leq i \leq n$.
\end{lm}
\begin{proof}[{\bf Proof.}]  Clear.
\end{proof}

We see that every arrangement can be ordered so to
satisfy the cut property (for example, the ordering of the hyperplanes in figure \ref{arrgt} satisfies the cut property). Indeed,  Definition \ref{cut} turns out to describe the property we were seeking for.

\begin{rem}If the ordering $\AA=\{H_1,\ldots,H_n\}$ satisfies the
cut property with respect to the chamber $B$, then for every $j=1,\ldots,n$ there is a canonical
choice of a base region in $(\AA_j)''$: $$B_j'':= H_j\cap B_{j-1}.$$
Moreover, the induced ordering of $(\AA_j)''$ satisfies the cut
property with respect to $B_j''$.
\end{rem}


\begin{df} Let $\AA:=\{H_1,\ldots, H_n\}$ be ordered such that $H_n\in\WW_B$. With the Notations of \ref{not2} we define: $$\tp:=\tp_B(\AA),\quad\quad\tp':=\tp_{B'}(\AA'),\quad\quad \tp'':=\tp_{B''}(\AA'').$$

 Moreover, let $\B'$ (or $\B'(\AA)$ if specification is needed) denote the set of all chambers of
$\AA'$ that are `cut' by $H_n$. Every $C'\in\B'$ contains therefore two
chambers $C^{\downarrow}\tldot_B C^{\uparrow}$ of $\tp$. Define
$$\Bu:=\{C^{\uparrow}\mid C\in \B\},\quad\Bd:=\{C^{\downarrow}\mid
C\in\B\},$$ 
$$\U:=\tp'\setminus \B',\quad\quad \B'':=\{H_n\cap C\mid C\in\B'(\AA)\}.
$$ 
\end{df}

\begin{rem}\label{bijections} Clearly,
$$\tp=\U\uplus\Bu\uplus\Bd,\quad\quad \tp'=\U\uplus\B',\quad\quad\tp''=\B'',$$
with the evident order preserving bijections:
$$\beta': \B'\rightarrow \Bd,\quad\quad\beta'':\Bu\rightarrow\B''.$$
\end{rem}

 We want to describe a particular linear extension of $\tp$ that allows us to explicitly index the critical elements of the associated acyclic matchings with the no broken circuit sets of the arrangement. We will make use of an indexing of the chambers of $\AA$ by {\em nbc} sets that is inspired by a result of Jewell and Orlik \cite{JO}.

\begin{df}[see Section 3.4 of \cite{JO}]\label{eta} 
Consider an ordering $\AA=\{H_1,\ldots, H_n\}$ that
  satisfies the cut property with respect to the chamber $B$  and keep
  the notations introduced above. We  define a map
  $$\eta:\tp_B(\AA)\rightarrow \mathcal{P}(\AA)$$ recursively in the
  number of elements of $\AA$ as
  follows:

 $\bullet$ If $\AA=\{H_1\}$, let $\eta_1(H_1^+):=\emptyset$
  and $\eta_1(-H_1^+):=\{H_1\}$.

 $\bullet$ Let $\AA=\{H_1,\ldots,H_n\}$ with $n>1$ and
suppose we are able to define such functions for every arrangement of
cardinality at most $n-1$. In particular the functions $\eta'$
and $\eta''$ associated to $\AA'$, $\AA''$ are defined. Then, for $C\in\tp(\AA)$ we define 
$$\eta(C):=\left\{\begin{array}{ll}
\eta'(C) &\textrm{ if } C\in\U\cup\Bd\\
\big\{\min\{H\in\AA \mid H\cap H_n= L\}\,\big\vert\, L\in\eta''(\beta''(C))\big\}
&\textrm{ if } C\in\Bu
\end{array}\right.  $$
\noindent where we slightly abused notation in implicitly identifying
$\tp'$ 
with $\U\cup\Bd$ using the bijection $\beta'$ of Definition \ref{bijections}.
\end{df}

In particular, for $C\in\B'(\AA)$ we have $\eta(C^\downarrow)=\eta'(C)$
and a natural bijective correspondence between $\eta(C^\uparrow)$ and
$\eta''(C\cap H_n)\cup\{H_n\}$. The map $\eta$ was
introduced in \cite{JO} as a bijection between no-broken circuit sets
and chambers of the arrangement, as we state in the following lemma.

\begin{lm}[see Lemma 3.14 of \cite{JO}]\label{jobij} 
The map $\eta$ is a bijection $\tp(\AA)\rightarrow\nbc(\AA)$ with
 $\eta(B)=\emptyset$.
\end{lm}

\begin{figure}[h]
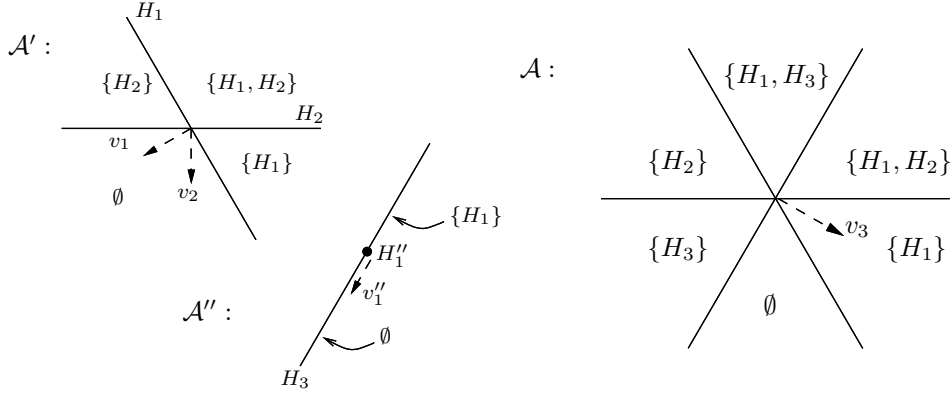

  \begin{picture}(0,0)%
    \includegraphics{nbcbj_scale_scaled.pstex}%
  \end{picture}%
  \input{nbcbj_scale_scaled.pstex_t}%
  
\caption{The last step in the inductive construction of $\eta$ for the
arrangement given on the left of Figure \ref{arrgt}, where we see that $\Bu=\{C_3,C_5\}$, $\Bd=\{C_2,C_4\}$, $\U=\{C_1,C_6\}$. For every chamber
$C$, the set $\eta(C)$ is written inside $C$ to show the bijective correspondence.}\label{nbcbi}
\end{figure}

\newcommand{\hB}{\widehat{B}}
\newcommand{\hAA}{\widehat{\AA}}
\newcommand{\hphi}{\widehat{\phi}}
\newcommand{\heta}{\widehat{\eta}}
\newcommand{\uC}{C^\uparrow}\newcommand{\dC}{C^\downarrow}
\newcommand{\uK}{K^\uparrow}\newcommand{\dK}{K^\downarrow}

 As a first step let us prove a technical property that derives from
 our particular choice of the ordering of the hyperplanes.

\begin{lm}\label{inters_nbc} Let $\AA=\{H_1,\ldots ,H_n\}$ be an arrangement of linear
  real hyperplanes and $B$ a chamber of $\AA$. Suppose that the
  ordering of the hyperplanes satisfies the cut property with respect
  to $B$. 
Then
  $$\bigcap\eta'(C)\cap H_n =\bigcap\eta''(C\cap H_n)\quad\quad\forall C\in\B'(\AA).$$
\end{lm}
\begin{proof}[{\bf Proof.}]  Again, we argue recursively on the number of hyperplanes of $\AA$. If $\AA=\{H_1\}$ there is nothing to prove. So let $\AA=\{H_1,\ldots,H_n\}$ with $n>1$ and suppose that the ordering satisfies the cut property with respect to the chamber $B$. Let $\hAA:=\AA\setminus\{H_{n-1}\}$. Clearly the induced ordering on $\hAA$ satisfies the cut property with respect to $\hB:=\bigcap_{j\neq n-1}H_j^+$ and thus, by induction, the claim holds and ensures 
$$\bigcap\heta'(C)\cap H_n=\bigcap \heta''(C\cap H_n) \quad\quad \forall C\in\B'(\hAA).$$
Also, the induction hypothesis applies to the arrangement $\AA''$ with
respect to the induced order and the chamber $B''=B\cap H_n$;
thus, if we define $L:=H_n\cap H_{n-1}$,  {\em when there is no $j<n-1$ with $H_J\supset L$}
we can write $$\bigcap\nu'(C)\cap L=\bigcap\nu''(C\cap L)\quad\quad\forall
C\in\B'(\AA''),$$
\noindent where $\nu,\nu',\nu'' $ are the maps obtained by applying
Definition \ref{eta} to $\AA''$. 
Finally, let us denote by $\mu,\mu',\mu''$ the maps associated to
$\AA'=\{H_1,\ldots , H_{n-1}\}$. We know that the order induced on
$\AA'$ satisfies the cut property with respect to the unique chamber
$B'\supset B$ and thus, by induction,
$$\bigcap\mu'(C)\cap H_{n-1} =\bigcap\mu''(C\cap H_{n-1})\quad\quad
\forall C\in\B'(\AA'). $$

We would like to point out the following (tautological) relations:
$$ \mu=\eta',\quad \heta'=\mu', \quad \heta''=\nu',\quad
\nu=\eta''. $$

Now we proceed with the proof. Let $\AA$ be as above, and choose
$C\in\B'(\AA)$.
It is easy to see that if $C\subset H_{n-1}^+$ or if
$H_{n-1}$ is not a wall of $C$, then the claim holds because it holds
for $\hAA$.

So suppose that $H_{n-1}$ is a wall of $C$ and that $C\not\subset
H_{n-1}^+$. Then we have
$$\eta'(C)=\mu(C)=\{H_{n-1}\} \cup \mu''(C\cap H_{n-1})$$

and 
$$\eta''(C\cap H_n)=\left\{\begin{array}{ll}
\heta''(C\cap H_n) &\textrm{if there is }j<n-1\textrm{ with } L\subset
H_j,\\
\{L\}\cup\nu''((C\cap H_n) \cap H_{n-1})&\textrm{else.}
\end{array}\right.$$

Moreover, we can write
$$\begin{array}{rl}
\bigcap\eta'(C)\cap H_n & =  \bigcap \big[\{H_{n-1}\}\cup\mu''(C\cap
H_{n-1})\big]\cap H_n\\
 &= \bigcap \mu'(C)\cap H_{n-1} \cap H_n = \bigcap \heta'(C)\cap H_{n}\cap H_{n-1}\\
 &= \bigcap\heta''(C\cap H_n) \cap H_{n-1}.
\end{array}$$

Since we know that $H_{n-1} \in \eta'(C)$, this implies
$\bigcap\eta'(C)\cap H_n =\bigcap \heta''(C\cap H_n)$. To conclude the
proof we distinguish two cases:

\noindent {\em Case 1.} If there is $j<n-1$ with $L\subset H_j$,  the
claim follows immediately, because then
$\heta''(C\cap H_n)=\eta''(C\cap H_n)$.

\noindent {\em Case 2.} If there is no such $j$, then the induction hypothesis
applies to $\nu$ and gives
$$\bigcap\heta''(C\cap H_n) \cap H_{n-1}=\bigcap\nu'(C\cap H_n)\cap H_{n-1}  =\bigcap\nu''(C\cap H_n\cap
H_{n-1})=\eta''(C), $$
where the last inequality holds because every element of
$\nu''(C\cap H_n\cap H_{n-1})$ is contained in $L$.

Thus, in any case the claim holds.
\end{proof}

Now the idea is to consider a linear extension that behaves well under `taking $\AA'$ and $\AA''$'.

\begin{df}\label{lex} For every $H\in \AA$ let $H^+$ denote the open halfspace
  that is bounded by $H$ and contains the base chamber $B$. To every
  $C\in \tp$ we associate an array $\sigma(C):=(\sigma_1(C),\ldots,
  \sigma_n(C))$ by setting $\sigma_i(C)=0$ if $C\subset H^+_i$, and $\sigma_i(C)=1$ else.

We denote by $\tlexl$ (or $\tlexl_{\AA,B}$ when specification is needed) the total order on $\tp$ induced by the lexicographic ordering of the corresponding arrays.
\end{df}

\begin{ex} The linear extension of example \ref{exlinex} translates into 
$$(0,0,0)\texl (0,0,1) \texl (0,1,1)\texl (1,0,0)\texl (1,1,0)\texl (1,1,1) $$ and is therefore $\tlexl_{\AA}$ for the arrantement $\AA$ of Figure \ref{arrgt}.
\end{ex}

\begin{rem} In the language of oriented matroids the above definition just fixes the acyclic orientation associated with the tope $B$ and then associates to every tope its signed covector.
\end{rem}

\begin{lm}\label{lessicogr} The ordering $\tlexl_{\AA,B}$ is a linear extension of $\tp_B(\AA)$, and satisfies:\begin{itemize}
\item[(1)] the ordering of $\tp'$ induced via the maps $\delta$, $\beta'$, $\gamma$ is $\tlexl_{\AA',B'}$.
\item[(2)] the ordering of $\tp''$ induced via the map $\beta''$ is $\tlexl_{\AA'', B''}$.
\end{itemize}
\end{lm}

\begin{proof}[{\bf Proof.}] 
We have to show that if $C\tl_B C'$, then $C\tlexl C'$. But the former means $S(B,C)\subset S(B,C')$: thus, $\sigma(C')$ is obtained from $\sigma(C)$ by switching from $0$ to $1$ the entries corresponding to the elements of $S(C,C')$, and $\tlexl$ is therefore a linear extension. 
Item (1) is easy to see. For (2), recall that every hyperplane of $\AA''$ corresponds to a codimension $2$ subspace of $\AA$ and gets the number of the smallest $i<n$ such that $H_i$ contains the subspace. 
\end{proof}

The next step will be to prove that the critical cells of the acyclic
matching of Proposition \ref{maxmat} are completely determined by the
associated chamber, provided that the chosen linear extension is the
one associated via Definition \ref{defX} to an ordering of the
hyperplanes that satisfies the cut property.

 We will
show that, for every base chamber $B$ and every ordering of $\AA$
satisfying the cut property with respect to $B$, $\eta(C)$ is a basis
of the flat $X_C$ if the chosen linear extension of $\tp_B(\AA)$ is
the one of Definition \ref{lex}.

\begin{thm}\label{corresp}
Let the ordering $\{H_1,\ldots,H_n\}$ of $\AA$ satisfy the cut property with respect to the chamber $B$ and consider the linear extension
$\tlexl$ of $\tp_B$. We have $$X_C=\bigcap\eta(C).$$ 
\end{thm}
   
\begin{proof}[{\bf Proof.}] 
Again, the claim is trivial if $¦\AA¦=1$. So let $n:=\vert \AA\vert >1$
and suppose that the claim holds for every arrangement of at most
$n-1$ hyperplanes (and thus, in particular, for $\AA'$ and $\AA''$).

Given $C\in \tp(\AA)$, let $$Y_C:=\bigcap\eta(C).$$

We are going to prove that $Y_C$ satisfies \ref{tec_lm}.(1) and \ref{tec_lm}.(2).


It is easily seen that this is true if $H_n\in S(B,C)$, because
the above properties hold for $\AA'$ and depend only on the position
of the flat with respect to the union of the chambers $K$ that come
before $C$. In fact, the chosen linear extension is such
that the union of all $K\tlexl C$ equals (as a subset of $\mathbb{R}^d$) the union of
the chambers that come before $C'$  with respect to the ordering
$\tlexl_{\AA',B'}$ (recall that $C'$ is the unique chamber of $\AA'$ containing
$C$).

So let $C\in\Bu$ and recall that by definition we have
$$\eta(C)=\{H_n\}\cup \big\{\min\{H\in\AA \mid H\cap H_n= L\}\,\big\vert\,
L\in\eta''(C\cap H_n)\big\}.$$

We now have to check the properties of Definition \ref{tec_lm}.

\noindent {\em \underline{\ref{tec_lm}.(1):}  $\,\supp(Y_C)\cap S(C,K)\neq \emptyset$ for all
  $K\tlexl C$.}\\ This assertion is clear if $H_n\in S(B,K)$, since then
  $H_n\in S(C,K)\cap \supp(Y_C)$. On the other hand, if $H_n\not\in S(B,K)
  \U$, then we know that $S(C,K)\cap\supp(\bigcap\eta'(C'))\neq \emptyset$ by
  induction hypothesis. But Lemma \ref{inters_nbc} allows us to write
$$Y_C=\bigcap\eta(C)=\bigcap\eta''(C\cap H_n)= \bigcap\eta'(C)\cap
H_n, $$
whence $\supp(Y_C)\supseteq \supp(\bigcap\eta'(C))$, and the claim
follows.\\
\noindent {\em \underline{\ref{tec_lm}.(2):} For every flat $Z\not\geq Y_C$
   in $\LL(\AA)$ there
  is a chamber  $K\tlexl C$ such that $\supp(Z)\cap S(C,K)=\emptyset$.}\\ Clearly if
  $H_n\not \in\supp(Z)$, we are easily done by taking
  $K=(C')^\downarrow$ so that $S(C,K)=\{H_n\}$. We are left with the
  case where
  $H_n\in\supp(Z)$. Then $Z\not\geq \bigcap\eta''(C'')$ in
  $\LL(\AA'')$ - recall Lemma \ref{inters_nbc} and that $C'':=C'\cap H_n$ - and by induction
  hypothesis we know that there is $K''\tlexl_{\AA'',B''}C''$ with no
  hyperplane of $\AA''$ containing $Z$ and separating $K''$ from $C''$. Now let $K$ be the chamber of $\AA$ that is `just above' (or: the
  preimage with respect to $\beta''^{-1}$ of) $K''$ (so that $K\tlexl
  C$ by Lemma \ref{lessicogr}). For every $H\in S(C,K)$,
  $H\cap H_n$ separates $C''$ from $K''$ in $\AA''$. Thus, if there
  were $H\in\supp(Z)\cap S(C,K)$, then there would be $L:=H\cap
  H_n\in\supp''(Z)$ separating $C''$ from $K''$ (where $\supp''(Z)$ is
  naturally defined as $\{L\in\AA''\mid Z\subset L\}$) - a contradiction. 
\end{proof}

We can now summarize our results leaving the greatest generality in
the attempt to approach the greatest {\em naturality}. The proof is an
easy combination of Proposition \ref{maxmat}, Theorem \ref{corresp}, Remark \ref{grezzo} and Corollary \ref{ciliegina}.

\begin{prop}\label{result} Let $\AA$ denote a real arrangement of linear
  hyperplanes and choose a chamber $B\in\tp(\AA)$. Every ordering of
  $\AA$ that satisfies the cut property with respect to $B$ gives rise to a bijection $\eta$ between chambers and nbc-sets as
in Definition \ref{eta} and to 
an acyclic matching of the Salvetti complex which critical cells are precisely those of the form
$$\big\langle \,\bigcap\eta(C)\cap C,\; C \,\big\rangle.$$

In particular, the resulting CW-complex has one cell of dimension
$\vert \eta(C)\vert$ for every  $C\in\tp(\AA)$.  
\end{prop}

\begin{ex} By comparing Figure \ref{arrgt} with Figures \ref{figurone}, \ref{salposf} and
  \ref{nbcbi} one sees immediately the claimed correspondence: $$\begin{array}{llll}
\eta(C_1)=\emptyset, & \bigcap\emptyset = \mathbb{R}^d=\hat{0}=X_{C_1}, & \mathbb{R}^d\cap C_1 = C_1, & \langle C_1, C_1\rangle \textrm{ is critical;}\\
\eta(C_2)=\{H_3\}, & \bigcap \{H_3\} = H_3 = X_{C_2}, & H_3\cap C_2=F_1, &\langle F_1, C_2 \rangle \textrm{ is critical;}\\
\eta(C_3)=\{H_2\}, & \bigcap \{H_2\} = H_2 = X_{C_3}, & H_2\cap C_3=F_6, &\langle F_6, C_3 \rangle \textrm{ is critical;}\\
\eta(C_4)=\{H_1\}, & \bigcap \{H_1\} = H_1 = X_{C_4}, & H_1\cap C_4=F_2, &\langle F_2, C_4 \rangle \textrm{ is critical;}\\
\eta(C_5)=\{H_1,H_3\}, & H_1\cap H_3 = P = X_{C_5}, & P\cap C_5 = P, & \langle P, C_5\rangle \textrm{ is critical;}  \\
\eta(C_6)=\{H_1,H_2\}, & H_1\cap H_2 = P = X_{C_6}, & P\cap C_6 = P, & \langle P, C_6\rangle \textrm{ is critical;}
\end{array}
$$
and there are no further critical cells.
\end{ex}

\begin{rem} The importance of the chambers in the above
  characterization of the critical cells is mainly to give the order
  along which we decompose the Salvetti complex. It is now natural to
  ask if such ordering can be defined purely in terms of the
  no-broken-circuit sets. This would actually allow to describe the
  situation without referring to the geometry of $\mathbb{R}^d$. However, this task might be particularly subtle:
  for instance, compare the arrangement of Coxeter type $A_2$ and the
  coordinate arrangement in $\mathbb{R}^3$ (let us call it $K_3$). Up to symmetry, in both
  cases there is only one linear ordering induced on the families of
  no-broken-circuit sets:
$$A_2:\quad\emptyset,\quad\{3\},\quad\{2\},\quad\{1\},\quad\{1,2\},\quad\{1,3\},$$$$K_3:\quad\emptyset,\quad\{3\},\quad\{2\},\quad\{2,3\},\quad\{1\},\quad\{1,3\},\quad\{1,2\},\quad\{1,2,3\}$$
(where we wrote $j$ for $H_j$) and we see that $\{1,2\}$ and $\{1,3\}$ are switched in the two
orderings. This seems to indicate that one should consider also some
`global' property of the lattice, other than just examining the
no-broken-circuit sets. 
\end{rem}

\bibliographystyle{amsplain}
\bibliography{bibmorse}

\end{document}